\documentclass[11pt]{amsart}
\usepackage[pdftex,margin=1.1in, marginpar=.8in]{geometry}

\usepackage{amsmath}
\usepackage{amsfonts}
\usepackage{amssymb}
\usepackage{amsthm}

\usepackage[colorlinks=true,urlcolor=blue,linkcolor=blue,citecolor=blue,pagebackref,breaklinks]{hyperref}
\usepackage{cite}

\usepackage{mathrsfs}
\usepackage{bm}
\usepackage{dsfont}
\usepackage{upgreek}
\usepackage{xcolor}
\usepackage{bbold}

\usepackage{graphicx}
\usepackage[all,cmtip]{xy}
\usepackage{blkarray}
\usepackage{array}
\usepackage{todonotes}
\usepackage{enumitem}  

\setlist[enumerate,1]{label = (\arabic*),itemsep=.5em, topsep=.5em}

\usepackage{leftidx}
\usepackage{mathdots}
\usepackage{arydshln}
\usepackage{stmaryrd}  
\SetSymbolFont{stmry}{bold}{U}{stmry}{m}{n}

\usepackage{tikz-cd}

\newenvironment{bsm}
{\left[\begin{smallmatrix}}
{\end{smallmatrix}\right]}

\numberwithin{equation}{subsection}

\newtheorem{thm}{Theorem}[subsection]
\newtheorem{prop}[thm]{Proposition}

\theoremstyle{remark}
\newtheorem{rmk}[thm]{Remark}

\newcommand{\mr}[1]{\mathrm{#1}}

\newcommand{\bA}{\mathbb{A}}

\newcommand{\bC}{\mathbb{C}}

\newcommand{\bI}{\mathbb{I}}

\newcommand{\bQ}{\mathbb{Q}}

\newcommand{\bS}{\mathbb{S}}
\newcommand{\bT}{\mathbb{T}}
\newcommand{\bU}{\mathbb{U}}

\newcommand{\bZ}{\mathbb{Z}}

\newcommand{\bba}{\mathbb{a}}
\newcommand{\bbb}{\mathbb{b}}
\newcommand{\bbc}{\mathbb{c}}

\newcommand{\cB}{\mathcal{B}}
\newcommand{\cC}{\mathcal{C}}
\newcommand{\cD}{\mathcal{D}}
\newcommand{\cE}{\mathcal{E}}
\newcommand{\cF}{\mathcal{F}}

\newcommand{\cH}{\mathcal{H}}

\newcommand{\cK}{\mathcal{K}}
\newcommand{\cL}{\mathcal{L}}
\newcommand{\cM}{\mathcal{M}}
\newcommand{\cN}{\mathcal{N}}
\newcommand{\cO}{\mathcal{O}}
\newcommand{\cP}{\mathcal{P}}

\newcommand{\cS}{\mathcal{S}}

\newcommand{\cW}{\mathcal{W}}

\newcommand{\fa}{\mathfrak{a}}

\newcommand{\fp}{\mathfrak{p}}

\newcommand{\fx}{\mathfrak{x}}

\newcommand{\fz}{\mathfrak{z}}

\newcommand{\fA}{\mathfrak{A}}
\newcommand{\fB}{\mathfrak{B}}

\newcommand{\fD}{\mathfrak{D}}

\newcommand{\sB}{\mathscr{B}}
\newcommand{\sC}{\mathscr{C}}

\newcommand{\sL}{\mathscr{L}}

\newcommand{\sS}{\mathscr{S}}
\newcommand{\sT}{\mathscr{T}}

\DeclareMathAlphabet{\mathpzc}{OT1}{pzc}{m}{it}

\newcommand{\pzV}{\mathpzc{V}}

\newcommand{\bfP}{\mathbf{P}}

\newcommand{\bbeta}{\bm{\beta}}

\newcommand{\utau}{\protect\underline{\tau}}

\DeclareMathOperator{\GL}{GL}
\DeclareMathOperator{\GSp}{GSp}

\DeclareMathOperator{\GU}{GU}

\DeclareMathOperator{\U}{U}

\DeclareMathOperator{\SL}{SL}
\DeclareMathOperator{\Sp}{Sp}

\newcommand{\cont}{\mathrm{cont}}

\newcommand{\diag}{\mathrm{diag}}

\newcommand{\Meas}{\mathpzc{M}eas}
\newcommand{\Nm}{\mathrm{Nm}}
\newcommand{\ord}{\mathrm{ord}}

\newcommand{\Tr}{\mathrm{Tr}}
\newcommand{\triv}{\mathrm{triv}}
\newcommand{\tor}{\mathrm{tor}}

\newcommand{\val}{\mathrm{val}}
\newcommand{\vol}{\mathrm{vol}}

\newcommand{\Sieg}{\mr{Sieg}}

\DeclareMathOperator{\Gal}{Gal}
\DeclareMathOperator{\Her}{Her}
\DeclareMathOperator{\Hom}{Hom}

\DeclareMathOperator{\Spec}{Spec}

\DeclareMathOperator{\Sym}{Sym}

\newcommand{\adic}{\text{-}\mathrm{adic}}
\newcommand{\qexp}{q\text{-}\mathrm{exp}}

\newcommand{\lhra}{\ensuremath{\lhook\joinrel\longrightarrow}}
\newcommand{\lra}{\longrightarrow}
\newcommand{\ra}{\rightarrow}
\newcommand{\hra}{\hookrightarrow}

\newcommand{\wh}{\widehat}
\newcommand{\wt}{\widetilde}

\newcommand{\llb}{\llbracket}
\newcommand{\rrb}{\rrbracket}

\newcommand{\bid}{\mathbf{1}}

\newdir{d}{{\subset}}   

\newcommand{\genlegendre}[4]{%
  \genfrac{(}{)}{}{#1}{#3}{#4}%
  \if\relax\detokenize{#2}\relax\else_{\!#2}\fi
}

\makeatletter
\def\l@section{\@tocline{1}{0pt}{1pc}{}{}}
\def\l@subsection{\@tocline{2}{0pt}{1pc}{4.6em}{}}
\def\l@subsubsection{\@tocline{3}{0pt}{1pc}{7.6em}{}}
\renewcommand{\tocsection}[3]{%
  \indentlabel{\@ifnotempty{#2}{\makebox[2.3em][l]{%
    \ignorespaces#1 #2.\hfill}}}#3}
\renewcommand{\tocsubsection}[3]{%
  \indentlabel{\@ifnotempty{#2}{\hspace*{2.3em}\makebox[2.3em][l]{%
    \ignorespaces#1 #2.\hfill}}}#3}
\renewcommand{\tocsubsubsection}[3]{%
  \indentlabel{\@ifnotempty{#2}{\hspace*{4.6em}\makebox[3em][l]{%
    \ignorespaces#1 #2.\hfill}}}#3}
\makeatother

\newcommand{\mf}{f}
\newcommand{\dsec}{{\tt f}}

\newcommand{\Bes}{\cB}
\newcommand{\bes}{B}
\newcommand{\Whi}{\cW}
\newcommand{\whi}{W}

\newcommand{\alphaS}{\alpha_{\mathbb{S}}}
\newcommand{\alphabS}{\bar{\alpha}_{\mathbb{S}}}
\newcommand{\Schw}{\Phi}
\newcommand{\sPi}{\scriptscriptstyle{\Pi}}
\newcommand{\sLambda}{\scriptscriptstyle{\Lambda}}
\newcommand{\disc}{\mathrm{disc}}

\newcommand{\bLam}{\boldsymbol\Lambda}
\newcommand{\bUps}{\boldsymbol\Upsilon}

\newcommand{\bdot}{\boldsymbol{\cdot}}

\newcommand{\spe}{\mathrm{sp}}

\newcommand{\KL}{\mathrm{KL}}

\title{$p$-adic $L$-functions for $\GSp(4)\times\GL(2)$ II}
\author{Zheng Liu}
\address{University of California, Santa Barbara, , CA, United States}\email{\href{mailto:zliu@math.ucsb.edu}{zliu@math.ucsb.edu}}

\begin{document}

\maketitle

\begin{abstract}
We construct four-variable $p$-adic $L$-functions for cuspidal Hida families on $\GSp(4)\times\GL(2)$ and prove a complete interpolation formula. The  archimedean zeta integrals are computed by using a partial interpolation formula for the four-variable $p$-adic $L$-functions and some previously constructed $p$-adic $L$-functions.
\end{abstract}

\tableofcontents 
\numberwithin{equation}{subsection}

\section{Introduction}

In this paper, we generalize the construction of the cyclotomic-variable $p$-adic $L$-function for $\GSp(4)\times\GL(2)$ in \cite{pFL} to construct four-variable $p$-adic $L$-functions for Hida families on $\GSp(4)$ and $\GL(2)$, and complete the interpolation formula in {\it loc.cit} by calculating the archimedean integrals via $p$-adic interpolations. 

Fix an odd prime $p$ and an isomorphism $\bar{\bQ}_p\cong\bC$. Let $F$ be a sufficiently large finite extension of $\bQ_p$ and $\cO$ be its ring of integers. For $G=\GL(2),\GSp(4)$, let $\Lambda_G$  be the Iwasawa algebra for $G$ over $\cO$ defined in \eqref{eq:Iw-alg}, and  $\bT_{G,\ord}$ be the Hecke algebra acting on the $\Lambda_G$-module of Hida families on $G$ of tame level $K^p_G$  (chosen as in \S\ref{sec:setup}). Given geometrically irreducible components 
\begin{align*}
   \sC_1&\subset \Spec(\bT_{\GL(2),\ord}),
   &\sC_2&\subset\Spec(\bT_{\GSp(4),\ord}),
\end{align*} 
denote by $\bI_{\sC_1},\bI_{\sC_2}$ their coordinate rings and by $F_{\sC_1},F_{\sC_2}$ their functions fields. We construct the (imprimitive) $p$-adic $L$-function for $\sC_1,\sC_2$ and verify its full interpolation properties as predicted by Coates and Perrin-Riou \cite{CoPerrin,Coates} when the weight $l$ of the specialization of $\sC_1$ and the weight $(l_1,l_2)$ of the specialization of $\sC_2$ belong to the region
\begin{equation}\label{eq:wt-region}
    \frac{\min\{-l_1+l_2+l,\,l_1+l_2-l\}}{2}\geq 3,
\end{equation}
(which is the region (D) in the convention of \cite{LoRi}).

\begin{thm}\label{thm:main}
Given Hida families $\sC_1,\sC_2$ on $\GL(2),\GSp(4)$ and  the auxiliary data: 
\begin{enumerate}
\item[--] $\bS=\begin{bmatrix}\bba&\frac{\bbb}{2}\\ \frac{\bbb}{2}&\bbc\end{bmatrix}\in\Sym_2(\bQ)_{>0}$,
\item[--] a $p$-adically continuous Hecke character $\bLam:\cK^\times\backslash\bA^\times_{\cK,f}\ra \Lambda^\times_{\GSp(4)}$ with $\cK=\bQ(\sqrt{-\det\bS})$, such that $\bLam|_{\bA^\times_\bQ}=\omega_{\sC_2}$, the central character associated to $\sC_2$,
\item[--] a finite set $S$ of places of $\bQ$ containing $p,\infty$ such that everything is unramified outside $S$, (see \S\ref{sec:setup} for the precise condition on $S$),
\item[--] an open subgroup $U^p$ of $\hat{\bZ}^{p,\times}$ containing $\prod_{v\notin S} \bZ^\times_v$,
\end{enumerate}
taking $\beta_1\in\bQ_{>0},\beta_2\in\Sym_2(\bQ)_{>0}$, there exists a four-variable $p$-adic $L$-function
\begin{align*}
   \cL^S_{\sC_1,\sC_2,\beta_1,\beta_2}\in &\,\Meas\left(\bQ^\times\backslash\bA^\times_{\bQ,f}/U^p,\bI_{\sC_1}\wh{\otimes}\bI_{\sC_2}\right)\otimes_{\bI_{\sC_1}\wh{\otimes}\bI_{\sC_2}} (F_{\sC_1}\wh{\otimes}F_{\sC_2})\\
   &\,\cong (\bI_{\sC_1}\wh{\otimes}\bI_{\sC_2})\llb \bQ^\times\backslash\bA^\times_{\bQ,f}/U^p\rrb) \otimes_{\bI_{\sC_1}\wh{\otimes}\bI_{\sC_2}} (F_{\sC_1}\wh{\otimes}F_{\sC_2})
\end{align*}
satisfying the interpolation property: 
\begin{align*}
    \cL^S_{\sC_1,\sC_2,\beta_1,\beta_2}(\kappa,x)
    =&\,2^{-l-l_1-l_2}i^l\sum_{f\in \sS_{\GL(2),x}} \frac{f_{\bbc}f_{\beta_1}}{\bfP(f,f)}\sum_{\varphi\in \sS_{\GSp(4)},x} \frac{\bes^\dagger_{\bS,\Lambda}\left(\varphi\right)\varphi_{\beta_2}}{\bfP(\varphi,\varphi)} \\
   &\times E_\infty \Big(k+\frac{l+l_1+l_2}{2},\Pi_x\times\pi_x\times \chi\Big) E_p\Big(k+\frac{l+l_1+l_2}{2},\Pi_x\times\pi_x\times \chi\Big)\\
   &\times L^S\Big(k+\frac{l+l_1+l_2}{2},\Pi_x\times\pi_x\times \chi\Big),
\end{align*}
where
\begin{enumerate}[leftmargin=2em]
\item[--] $x\in \sC_1(\bar{\bQ}_p)\times\sC_2(\bar{\bQ}_p)$ is a point at which the weight projection map $\Lambda_{\GL(2)}\hat{\otimes}_\cO \Lambda_{\GSp(4)}\ra \bT_{\GL(2),\ord}\hat{\otimes}_\cO\bT_{\GSp(4),\ord}$ is \'{e}tale with arithmetic image
\[
    (\tau,(\tau_1,\tau_2))=((l,\xi),(l_1,l_2,\xi_1,\xi_2))
    \in \Hom_\cont\left(T^1_{\GL(2)}(\bZ_p)\times T^1_{\GSp(4)}(\bZ_p),\bar{\bQ}^\times_p\right),
\]
and $\kappa=(k,\chi)$ is an arithmetic point in $\Hom_\cont\left(\bQ^\times\backslash\bA^\times_{\bQ,f}/U^p,\bar{\bQ}^\times_p\right)$, (see \S\ref{sec:pchar} for some of the notations), such that 
\begin{equation}\label{eq:kl}
  -\frac{\min\{-l_1+l_2+l,\,l_1+l_2-l\}}{2}+2 \leq k+\frac{l_1+l_2+l}{2}\leq \frac{\min\{-l_1+l_2+l,\,l_1+l_2-l\}}{2}-1,
\end{equation}
(when $\min\{-l_1+l_2+l,\,l_1+l_2-l\}\geq 3$, $s=k+\frac{l+l_1+l_2}{2}$ for such $k$'s are all the critical points for the $L$-function $L(s,\Pi_x\times\pi_x\times\chi)$),
\item[--] $\sS_{\GL(2),x}$ (resp. $\sS_{\GSp(4),x})$ is an orthogonal basis  of the space spanned by ordinary cuspidal holomorphic forms on $\GL(2)$ of weight $l$ and tame level $K^p_{\GL(2)}$ (resp. $\GSp(4)$ of weight $(l_1,l_2)$ and tame level $K^p_{\GSp(4)}$) with nebentypus at $p$ given by \eqref{eq:pneben}, belonging to the Hecke eigenspace parameterized by $x$, 
\item[--] $\pi_x$ (resp. $\Pi_x$) is any unitary cuspidal irreducible automorphic representation of $\GL(2,\bA_\bQ)$ (resp. $\GSp(4,\bA_\bQ)$) inside the representation generated by $\sS_{\GL(2),x}$ (resp. $\sS_{\GSp(4),x}$) twisted by a real power of $|\det|$, 
\item[--] the factors $E_p,E_\infty$ are the modified Euler factor for $p$-adic interpolation (see \eqref{eq:Ep}\eqref{eq:Einf} for the precise formula),
\item[--] $f_\bbc,f_{\beta_1}$ (resp. $\varphi_{\beta_2}$) denotes the Fourier coefficient of $f$ indexed by $\bbc,\beta_1$ (resp. of $\varphi$ indexed by $\beta_2$), 
\item[--] $\Lambda$ is the classical Hecke character corresponding to the specialization of $\bLam$ at $(\tau_1,\tau_2)$, and $\bes^\dagger_{\bS,\Lambda}(\varphi)$ is the Bessel period with a modification at $p$ (see \eqref{eq:bes-dagger} and \cite[\S2.2.1]{pFL} for the precise definition).
\end{enumerate}
\end{thm}

We make some remarks on comparisons between our results and previous works on critical values of degree $8$ $L$-functions for $\GSp(4)\times\GL(2)$. There have been constructions of $p$-adic $L$-functions of one and three variables. For weights in the region \eqref{eq:wt-region} considered in this paper, a one-variable $p$-adic $L$-function ($l=l_l=l_2=-k-1$) is constructed in \cite{Agarwal}, and a three-variable $p$-adic $L$-function ($k+\frac{l_1+l_2+l}{2}=\frac{\min\{-l_1+l_2+l,\,l_1+l_2-l\}}{2}-1$) is constructed in \cite{LoRi}, and a one-variable cyclotomic $p$-adic $L$-function is constructed in \cite{pFL}. For weights in a different region where $-l_1+l_2+l\leq 1$, a one-variable cyclotomic $p$-adic $L$-function is constructed in \cite{LPSZ} and a three-variable $p$-adic $L$-function is constructed in \cite{LoZe}. (In all the previous constructions, the interpolations formulas are less complete than the one we prove in Theorem~\ref{thm:main}. They include unramified conditions at $p$ or conditions at ramified places away from $p$ or uncomputed local zeta integrals.) The constructions in \cite{Agarwal,pFL} start with the same automorphic integral as we utilize in this paper, {\it i.e.} Furusawa's formula recalled in \S\ref{sec:Furusawa}. The constructions in \cite{LPSZ,LoZe,LoRi} start with a different automorphic integral involving globally generic (non-holomorphic) automorphic forms on $\GSp(4)$ and Eisenstein series on $\GL(2)$. The constructions in  \cite{LPSZ,LoZe,LoRi} are motivated by studying Euler systems for $\GSp(4)\times\GL(2)$ constructed from Siegel units for modular curves. One major motivation for our construction is studying congruences between Yoshida lifts and other cuspidal automorphic representations on $\GSp(4)$.

Computations of local zeta integrals for Furusawa's formula are crucial for deducing algebraicity results on critical values for $L(s,\Pi\times\pi)$ and have been extensively studied  \cite{Furusawa,PiSch,Saha,Pitale,Saha-pullback,Morimoto-GSp4GL2-I,Morimoto-GSp4GL2-II}. Before \cite{Morimoto-GSp4GL2-II}, the arhcimedean zeta integrals have only been computed for holomorphic discrete series $\Pi_\infty$ of scalar weights. For general vector weights, the computation achieved in \cite{Morimoto-GSp4GL2-II} is up to $\bQ^\times$. Here, by utilizing the four-variable $p$-adic $L$-function for Hida families of Yoshida lifts and several previously constructed $p$-adic $L$-functions for Hida families (more precisely the Kubota--Leopoldt $p$-adic $L$-function \cite{KuLeo}, Rankin--Selberg $p$-adic $L$-function \cite{Hi88} and standard $p$-adic $L$-function for $\Sp(4)$ \cite{SLF}), we obtain a quantative result \eqref{eq:arch-zeta} on computing the archimdean zeta integrals when $\Pi_\infty$ is a holomorphic discrete series of general vector weights (for particular choices of test sections). (See the explanations of notations in Theorem~\ref{thm:pFL1} and \eqref{eq:Einf} for some notations in \eqref{eq:arch-zeta}.)

When carrying out the strategy of computing the archimedean zeta integrals via a comparison between the four-variable $p$-adic $L$-functions for the Yoshida lifts of two Hida families on $\GL(2)$, which interpolate special values of a product of Rankin--Selberg $L$-functions, and Hida's Rankin--Selberg $p$-adic $L$-functions, the key step is to compare the periods appearing in the interpolation formulas. The Petersson norm of the Yoshida lift is involved. One possible approach is to rewrite this Petersson norm in terms of the value at $s=1$ of the Rankin--Selberg $L$-function for the two modular forms used for the Yoshida lift. However, for our purpose, this approach requires precise formulas relating the Petersson norm and the Rankin--Selberg $L$-value when the components at $p$ of the automorphic representations of $\GL(2,\bA_\bQ)$  are principal series induced from sufficiently ramified characters, which are not currently available. (The cases treated in \cite{HsiehNamiInner} are those with  the component at $p$ unramified or Steinberg.)  We bypass this difficulty by employing an alternative approach making use of the $p$-adic standard $L$-functions for Hida families on $\Sp(4)$ constructed in \cite{SLF}.\\
 
\noindent{\bf Acknowledgement.} The author would like to thank Ming-Lun Hsieh for suggesting considering Furusawa's automorphic integral for $\GSp(4)\times\GL(2)$. During the preparation of this paper, the author was partially supported by the NSF grant DMS-2001527. 

\section{Notation and review of Hida theory and Furusawa's formula}

\subsection{Notation}

We fix an odd prime number $p$, an isomorphism $\bar{\bQ}_p\cong\bC$, and a sufficiently large  finite extension $F$ of $\bQ_p$. Denote by $\cO$ the ring of integers of $F$. 

We use $v$ to denote a place of $\bQ$. We fix the additive character 
\[
   \psi_{\bA_\bQ}=\bigotimes_v\psi_v:\bQ\backslash\bA\ra\bC^\times,
   \qquad \psi_v(x)=\left\{\begin{array}{ll}e^{-2\pi i\{x\}_v},& v\neq\infty\\e^{2\pi ix},&v=\infty\end{array},\right.
\] 
where $\{x\}_v$ is the fractional part of $x$.

Given a positive integer $n$,  define the algebraic group $\GSp(2n)$ over $\bZ$ as
\begin{align*}
   \GSp(2n,R)=\left\{g\in\GL(2n,R):\ltrans{g}\begin{bmatrix}0&\bid_n\\-\bid_n&0\end{bmatrix}g=\nu_g\begin{bmatrix}0&\bid_n\\-\bid_n&0\end{bmatrix}, \,\nu_g\in R^\times\right\}
\end{align*}
for all $\bZ$-algebra $R$. Given an imaginary quadratic field $\cK$, define the algebraic group $\GU(n,n)$ over $\bZ$ as
\begin{align*}
   \GU(n,n)(R)=\left\{g\in\GL(2n,\cO_\cK\otimes R):\ltrans{\bar{g}}\begin{bmatrix}0&\bid_n\\-\bid_n&0\end{bmatrix}g=\nu_g\begin{bmatrix}0&\bid_n\\-\bid_n&0\end{bmatrix}, \,\nu_g\in R^\times\right\},
\end{align*}
where for $\alpha\in\cK$, $\bar{\alpha}$ denotes its image under the nontrivial element in $\Gal(\cK/\bQ)$. In this paper, we will work with $\GSp(4),\GSp(2)=\GL(2),\GU(3,3),\GU(1,1)$. 

Fix the following maximal torus of $\GSp(2n)$, $\GU(n,n)$:
\begin{align*}
   T_{\GSp(2n)}&=\left\{\diag(a_1,\cdots,a_n,\nu a^{-1}_1,\cdots,\nu a^{-1}_n)\in\GSp(2n)\right\},\\
   T_{\GU(n,n)}&=\left\{\diag(\fa_1,\cdots,\fa_n,\nu \bar{\fa}^{-1}_1,\cdots,\nu \bar{\fa}^{-1}_n)\in\GU(n,n)\right\},
\end{align*}
and Siegel parabolic subgroup
\begin{align*}
   Q_{\GSp(2n)}&=\left\{\begin{bmatrix} A&B\\0&\nu \ltrans{A}^{-1}\end{bmatrix}\in \GSp(2n)\right\},
   &Q_{\GU(n,n)}&=\left\{\begin{bmatrix} \fA&\fB\\0&\nu \ltrans{\bar{\fA}}^{-1}\end{bmatrix}\in \GU(n,n)\right\}.
\end{align*}
Denote by $M_{\GSp(2n)}\subset Q_{\GSp(2n)},M_{\GU(n,n)}\subset Q_{\GU(n,n)}$ the Levi subgroup. Let
\begin{align*}
   T^1_{\GSp(2n)}&=T_{\GSp(2n)}\cap \Sp(2n),
   &T^1_{\GU(n,n)}&=T_{\GU(n,n)}\cap \U(n,n),\\
   M^1_{\GSp(2n)}&=M_{\GSp(2n)}\cap \Sp(2n),
   &M^1_{\GU(n,n)}&=M_{\GU(n,n)}\cap \U(n,n),
\end{align*}
where $\Sp(2n)$ (resp. $\U(n,n)$) is the subgroup of $\GSp(2n)$ (resp. $\GU(n,n)$) consisting of elements with similitude $1$. We identify $M^1_{\GSp(2n)},M^1_{\GU(n,n)}$ with $\GL(n),\GL(n)_{/\cK}$ via
\begin{align*}
   A&\longmapsto \begin{bmatrix} A\\&\ltrans{A}^{-1}\end{bmatrix},
   &\fA&\longmapsto \begin{bmatrix} \fA\\&\ltrans{\bar{\fA}}^{-1}\end{bmatrix}.
\end{align*}
Denote by 
\begin{align*}
U_{M^1_{\GSp(2n)}}&\subset M^1_{\GU(n,n)},
&U_{M^1_{\GU(n,n)}}&\subset M^1_{\GU(n,n)}
\end{align*} 
the subgroup whose elements are upper triangular with diagonal entries being $1$ under the above identification. Let
\begin{align*}
   U_{\GSp(2n)}&=\left\{\begin{bmatrix} A&B\\0&\nu \ltrans{A}^{-1}\end{bmatrix}\in Q_{\GSp(2n)}:A\in U_{M^1_{\GSp(2n)}}\right\},\\
   U_{\GU(n,n)}&=\left\{\begin{bmatrix} \fA&\fB\\0&\nu \ltrans{\bar{\fA}}^{-1}\end{bmatrix}\in Q_{\GU(n,n)}:\fA\in U_{M^1_{\GU(n,n}}\right\}.
\end{align*}

Given a character $\theta:\bQ^\times_p\ra \bC^\times$ or a character $\Theta:\cK^\times_p\ra\bC^\times$, we let 
\begin{align}
  \label{eq:char-circ} \theta^\circ&=\mathds{1}_{\bZ^\times_p}\cdot \theta,
   &\Theta^\circ&=\mathds{1}_{\cO^\times_{\cK,p}}\cdot \Theta.
\end{align}
Given a Hecke character $\Theta:\cK^\times\backslash\bA^\times_\cK\ra\bC^\times$, we let
\[
   \Theta_\bQ=\Theta|_{\bA^\times_\bQ}.
\]

\subsection{Review of Hida theory}\label{sec:Hida}
We recall some constructions and results in Hida theory for symplectic groups. (We will use it for $\GSp(4)$ and $\GSp(2)=\GL(2)$.) See \cite{HidaCon} or \cite[\S6.2]{SLF} for details.

For $G=\GSp(2n)$, define the Iwasawa algebra
\begin{align}
   \label{eq:Iw-alg}\wt{\Lambda}_G&=\cO\llb T^1_G(\bZ_p)\rrb,
   &\Lambda_G&=\cO\llb T^1_G(1+p\bZ_p)\rrb\cong\cO\llb T_1,T_2,\cdots,T_n\rrb.
\end{align}
Fix a neat open compact subgroup $K^p_G\subset G(\bA^p_{\bQ,f})$. Let $Y_G$ denote the Shimura variety for $G$ of level $K^p_G G(\bZ_p)$ defined over $\cO$. Let $\sT_{G,l,m}$ denote the $l$-th layer of the Igusa tower over $\bZ/p^m\bZ$, which is an $M^1_G(\bZ/p^l)$-\'{e}tale cover of the ordinary locus $Y_{G,\ord}$, and $\sT^\tor_{G,l,m}$ be a smooth partial toroidal compactification of $\sT_{G,l,m}$ (with respect to a chosen polyhedral cone decomposition) with boundary $C$. Put
\[
   V_{G,l,m}=H^0\left(\sT^\tor_{G,l,m},\cO_{\sT_{G,l,m}}(-C)\right)^{U_{M^1_G(\bZ_p)}},
\]
and 
\begin{align*}
   V_G&=\varprojlim_m\varinjlim_l V_{G,l,m},
   &\pzV_G&=\varinjlim_m\varinjlim_l V_{G,l,m}.
\end{align*}
The group $T^1_G(\bZ_p)$ naturally acts on these spaces, and they are all naturally modules over $\wt{\Lambda}_G$ and $\Lambda_G$.

For a tuple of integers $\underline{m}:m_1\geq m_2\geq\cdots\geq m_0\geq 0$ and $\nu\geq 0$ (corresponding to $\diag(p^{m_1+m_0},\dots,p^{m_n+m_0},p^{-m_1},\dots,p^{-m_n})$), an operator $U^G_{p,\underline{m},m_0}$ on $V_{G,l,m}$ is defined. We call these operators $\bU_p$-operators. Put $U^G_p=U^G_{p,\underline{m},0}$ with $\underline{m}=(n,n-1,\dots,1)$. As operators on $V_{G,l,m}$, the limit
\[
    e^G_\ord=\lim_{r\to \infty} \left(U^G_p\right)^{r!} 
\]
exists. Let
\begin{align*}
   V_{G,\ord}&= e^G_\ord V_G,
   &\pzV_{G,\ord}&= e^G_\ord \pzV_G,
   &\pzV^*_{G,\ord}&=\Hom(\pzV_{G,\ord},F/\cO).
\end{align*}
The $\tilde{\Lambda}_G$-module of Hida families of cuspidal $p$-adic automorphic forms on $G$ of tame level $K^p$ is defined to be
\[
   \cM_{G,\ord}=\Hom_{\wt{\Lambda}_G}(\pzV^*_G,\wt{\Lambda}_G).
\]
For each $p$-adic weight $\underline{\tau}\in \Hom_{\cont}\left(T^1_G(\bZ_p),\bar{\bQ}^\times_p\right)$, there is a natural map
\begin{equation}\label{eq:embd}
\begin{aligned}
     V_{G,\ord}[\utau]
     &\lra \Hom_{\cO}\left(\Hom(V_{G,\ord},\cO)/\cP_{\utau},\wt{\Lambda}_G/\cP_{\utau}\right)\\
     &\lra \Hom_{\cO}\left(\Hom(\pzV_{G,\ord},F/\cO)/\cP_{\utau},\wt{\Lambda}_G/\cP_{\utau}\right) \lra  \cM_{G,\ord}\otimes_{\wt{\Lambda}_G}  \wt{\Lambda}_G/\cP_{\utau},
\end{aligned}
\end{equation}
where $\cP_{\utau}$ is the prime ideal of $\wt{\Lambda}_G$ corresponding to $\utau$.

\begin{thm}
$\cM_{G,\ord}$ is free over $\Lambda_{G}$ of finite rank. The map \eqref{eq:embd} is an isomorphism. If $\utau$ is algebraic and sufficiently regular, then
\[
   V_{G,\ord}[\utau]=e^G_\ord H^0\left(Y^\tor_G,\omega_{\utau}(-C)\right),
\]
where $\omega_{\utau}$ is the automorphic vector bundle of weight $\utau$ over a toroidal compactification of $Y_G$. (The right hand side can be identified with the space of classical holomorphic automorphic forms on $G$ of weight $\utau$.)
\end{thm}


\subsection{Review of Furusawa's formula}\label{sec:Furusawa}
We quickly recall a modification of Furusawa's formula for $L$-functions for $\GSp(4)\times\GL(2)$. See \cite[\S2.1]{pFL} for details. Take 
\begin{align*}
  \bS=\begin{bmatrix}\bba&\frac{\bbb}{2}\\\frac{\bbb}{2}&\bbc\end{bmatrix}\in\Sym_2(\bQ)_{>0},
\end{align*}
and let  $\cK=\bQ(\sqrt{-\det \bS})$, and $\eta_{\cK/\bQ}:\bQ^\times\backslash\bA^\times_\bQ\ra\bC^\times$ be the quadratic character corresponding to $\cK/\bQ$. Let
\begin{align*}
   \alphaS&=\frac{\bbb+\sqrt{\bbb^2-4\bba\bbc}}{2\bbc},
      &\imath_\bS\left(\fz\right)
   =&\begin{bmatrix}\alphaS&1\\ \alphabS&1\end{bmatrix}^{-1} \begin{bmatrix}\fz\\&\bar{\fz}\end{bmatrix}\begin{bmatrix}\alphaS&1\\ \alphabS&1\end{bmatrix}
\end{align*}
for $\fz\in \cK\otimes_\bQ R$ with $R$ any $\bQ$-algebra.

Given a Hecke character $\Xi:\cK^\times\backslash\bA^\times_\cK\ra\bC^\times$, denote by $I_v(s,\chi,\Xi)$ the degenerate principal series on $\GU(3,3)(\bQ_v)$ consisting of smooth functions $\dsec_v(s,\chi,\Xi):\GU(3,3)(\bQ_v)\ra\bC$ such that 
\begin{align*}
   \dsec_v(s,\chi,\Xi)\left(\begin{bmatrix}\fA&\fB\\0&\fD\end{bmatrix}g\right)=\Xi_v(\det \fA)\chi_v(\det \fA\fD^{-1})|\det \fA\fD^{-1}|^{s+\frac{3}{2}}_v\,\dsec_v(s,\chi,\Xi)(g)
\end{align*}
for all $g\in\GU(3,3)(\bQ_v)$ and $\begin{bmatrix}\fA&\fB\\0&\fD\end{bmatrix}\in Q_{\GU(3,3)}(\bQ_v)$. The Siegel Eisenstein series associated to a section $\dsec(s,\chi,\Xi)\in I(s,\chi,\Xi)=\bigotimes'_v I_v(s,\chi,\Xi)$ is defined as
\[
    E^{\Sieg}(g;\dsec(s,\chi,\Xi))=\sum_{\gamma\in Q_{\GU(3,3)}(\bQ)\backslash \GU(3,3)(\bQ)}\dsec(s,\chi,\Xi)(\gamma g).
\]

Let $\pi$ be an irreducible cuspidal automorphic representation of $\GL(2,\bA_\bQ)$. By taking a Hecke character $\Upsilon:\cK^\times\backslash\bA^\times_\cK\ra\bC^\times$ with $\Upsilon_\bQ=\Upsilon|_{\bA^\times_\bQ}$ equal to the central character of $\pi$, for every $f\in\pi$, we can extend it to an automotphic form $f^\Upsilon$ on $\GU(1,1)$ by
\begin{equation}\label{eq:extGL2}
\begin{aligned}
   \mf^\Upsilon(\fa g)&=\Upsilon(\fa)f(g), &\fa\in\bA^\times_\cK,\,g\in\GL(2,\bA_\bQ).
\end{aligned}
\end{equation}
Then $\pi^{\Upsilon}=\{f^\Upsilon:f\in\pi\}$ is an irreducible cuspidal automorphic representation of $\GU(1,1)$. Denote the Whittaker period of $f\in \pi$ with respect to $\psi_{\bA_\bQ,\bbc}$ (defined as  $\psi_{\bA_\bQ,\bbc}(x)=\psi_{\bA_\bQ}(\bbc x)$) by $\whi_\bbc(f)$, and define the function $\Whi_\bbc(f)$ on $\GL(2,\bA_\bQ)$ as
\[
      \Whi_\bbc(f)(g)=\whi_\bbc(g\cdot f).
\]

Let $\Pi$ be an irreducible cuspidal automorphic representation of  $\GSp(4,\bA_\bQ)$, and $\Lambda:\cK^\times\backslash\bA^\times_\cK\ra\bC^\times$ be a Hecke character  such that $\Lambda_\bQ=\Lambda|_{\bA^\times_\bQ}$ equals the central character of $\Pi$. Then for $\varphi\in \Pi$, one can define its global Bessel period $\bes_{\bS,\Lambda}(\varphi)$ with respect to $\bS,\Lambda$ (and $\psi$) as in \cite[\S2.2.1]{pFL}. We also define the function $\Bes_{\bS,\Lambda}(\varphi)$ on $\GSp(4,\bA_\bQ)$ as
\begin{align*}
   \Bes_{\bS,\Lambda}(\varphi)(g)=\bes_{\bS,\Lambda}(g\cdot \varphi).
\end{align*}

Combining Furusawa's formula \cite{Furusawa} and Garrett's generalization of the doubling method \cite{GaKl} gives the following formula.
\begin{thm}\label{thm:Furusawa} 
Assume that $\Xi\Lambda^c\Upsilon^c=\triv$ and $S$ is a subset of places of $\bQ$ containing $\infty$ such that at all $v\notin S$,
\begin{enumerate}
\item[-] $\pi,\Pi,\Upsilon,\Xi,\chi,\dsec(s,\chi,\Xi)$ are all unramified  and $\varphi, f$ are spherical,
\item[-] $\bS=\begin{bmatrix}\bba&\frac{\bbb}{2}\\\frac{\bbb}{2}&\bbc\end{bmatrix}$ belongs to $M_2(\bZ_v)$ with $\bbc\in\bZ^\times_v$ and $\bbb^2-4\bba\bbc=\mr{disc}(\cK_v/\bQ_v)$.
\end{enumerate}
Then
\begin{equation*}
\begin{aligned}
   &\int_{[\GSp(4)\times_{\GL(1)}\GU(1,1)]} E^{\Sieg}\big(\imath(g,h);\dsec(s,\chi,\Xi)\big)\cdot \varphi(g) \cdot \mf^\Upsilon(h)\,\Xi^{-1}(\det h) \,dh\,dg\\
   =&\,\whi_\bbc(f)\cdot \bes_{\bS,\Lambda}(\varphi)\cdot d^S_3\left(s+\frac{1}{2},\Xi(\chi\circ\Nm)\right)^{-1}L^S\left(s+\frac{1}{2},\tilde{\Pi}\times\tilde{\pi}\times\chi\right)\\   
   &\times \prod_{v\in S} Z_v\Big(\dsec_v(s,\chi,\Xi),\Bes^{\Pi_v}_{\bS,\Lambda_v}(\varphi_v), \Whi^{\pi_v,\Upsilon_v}_\bbc(f_v)\Big)
\end{aligned}
\end{equation*}
with $\GSp(4)\times_{\GL(1)}\GU(1,1)=\left\{(g,h)\in \GSp(4)\times\GU(1,1):\nu_g=\nu_h\right\}$, and
\begin{equation*}
\begin{aligned}
   &Z_v\Big(\dsec_v(s,\chi,\Xi),,\Bes^{\Pi_v}_{\bS,\Lambda_v}(\varphi_v), \Whi^{\pi_v,\Upsilon_v}_\bbc(f_v)\Big)\\
   =&\,\Bes^{\Pi_v}_{\bS,\Lambda_v}(\bid_4)^{-1}\Whi^{\pi_v,\Upsilon_v}_\bbc(f_v)(\bid_2)^{-1}\int_{\big(R'_\bS\backslash\GSp(4)\times_{\GL(1)}\GU(1,1)\big)(\bQ_v)} \dsec_v(s,\chi,\Xi)\left(\cS^{-1}\imath(\eta_\bS\, g,h)\right)\\
   &\hspace{11em} \times \Bes^{\Pi_v}_{\bS,\Lambda_v}(\varphi_v)(g)
   \Whi^{\pi_v,\Upsilon_v}_\bbc(f_v)\left(\begin{bmatrix}0&1\\-1&0\end{bmatrix}h\right)\Xi^{-1}_v(\det h)\,dhdg,
\end{aligned}
\end{equation*}
where $\Bes^{\Pi_v}_{\bS,\Lambda_v}$ is the element corresponding to $\varphi$ in the local Bessel model of $\Pi_v$, $\Whi^{\pi_v,\Upsilon_v}_\bbc(f_v)$ is the extension to $\GU(1,1)(\bQ_v)$ via $\Upsilon_v$ of the element corresponding to $f$ in the local Whittaker model of $\pi_v$, and 
\begin{align*}
   \eta_\bS&=\begin{bmatrix}
  1\\ \alphaS&1\\&&1&-\alphabS\\&&&1
  \end{bmatrix},
   &\cS&=\begin{bmatrix}1\\&1\\&&1\\&&&1\\&&1&&1\\&1&&&&1\end{bmatrix},
\end{align*}
\[
   	d^S_3\big(s,\Xi(\chi\circ\Nm)\big)=\prod_{j=1}^3 L_S\left(2s+j,\Xi_\bQ\chi^2\eta^{n-j}_{\cK/\bQ}\right),
\]
$R'_\bS \subset \GSp(4)\times_{\GL(1)}\GU(1,1)$ is the subgroup
\[
   \left\{\left(\begin{bmatrix}\imath_\bS(\fz)\\&\ltrans{\imath_\bS(\bar{\fz})}\end{bmatrix}\begin{bmatrix}\bid_2&X\\&\bid_2\end{bmatrix},\,\fz\cdot\bid_2\right):\fz\in \mr{Res}_{\cK/\bQ}\GL(1),\,X\in\Sym_2\right\}.
\]
\end{thm}

In \cite{pFL}, a one-variable cyclotomic $p$-adic $L$-function for $\Pi\times\pi$ is constructed by using the above integral and  interpolating Siegel Eisenstein series when $s,\chi$ vary. In next section, we let $\pi,\Pi$ also vary in Hida families, and construct a four-variable Siegel Eisenstein family.

\section{Four-variable $p$-adic $L$-function for $\GSp(4)\times\GL(2)$}

\subsection{The setup}\label{sec:setup}
Let $S$ be a finite set of places of $\bQ$ containing $p,\infty$ and $U^p$  be an open subgroup of $\hat{\bZ}^{p,\times}$ containing $\prod_{v\notin S} \bZ^\times_v$. We fix tame level groups\begin{align*}
   K^p_{\GL(2)}&=\prod_{v\notin S}\GL(2,\bZ_v)\times\prod_{v\in S-\{\infty\}} K_{\GL(2),v}, &&\text{with }\begin{bmatrix}1&\bZ_v\\&1\end{bmatrix}\subset K_{\GL(2),v},\\
   K^p_{\GSp(4)}&=\prod_{v\notin S}\GSp(4,\bZ_v)\times\prod_{v\in S-\{\infty\}} K_{\GSp(4),v},&&\text{with }\begin{bmatrix}\bid_2&\Sym_2(\bZ_v)\\&\bid_2\end{bmatrix}\subset K_{\GSp(4),v}.
\end{align*}
Let $\bT_{\GL(2),\ord}$ (resp. $\bT_{\GSp(4),\ord}$) be the Hecke algebras acting on the $\tilde{\Lambda}_{\GL(2)}$-module $\cM_{\GL(2),\ord}$ of Hida families  of tame level $K^p_{\GL(2)}$ (resp. $\tilde{\Lambda}_{\GSp(4)}$-module $\cM_{\GSp(4),\ord}$ of Hida families  of tame level $K^p_{\GSp(4)}$) (consisting of all the spherical Hecke operators away from $S$ and the $\bU_p$-operators at $p$). Assume that we are given the following data:


\begin{enumerate}
\item[--] a geometrically irreducible component  $\sC_1\subset \Spec(\bT_{\GL(2),\ord})$ with central character 
\[
   \omega_{\sC_1}:\bQ^\times\backslash\bA^\times_{\bQ,f}\lra \Lambda^\times_{\GL(2)},
\]
\item[--] a geometrically irreducible component  $\sC_2\subset \Spec(\bT_{\GSp(4),\ord})$ with central character
\[
   \omega_{\sC_2}:\bQ^\times\backslash\bA^\times_{\bQ,f}\lra \Lambda^\times_{\GSp(4)},
\]
\end{enumerate}
plus the auxiliary data:
\begin{enumerate}
\item[--] an imaginary quadratic field $\cK$ and a positive definite symmetric form 
\[\bS=\begin{bmatrix}\bba&\frac{\bbb}{2}\\ \frac{\bbb}{2}&\bbc\end{bmatrix}\in\Sym_2(\bQ)_{>0}
\] 
with $\cK=\bQ(\sqrt{-\det\bS})$ such that for all $v\notin S$,
\begin{align*}
   &\bS\in \Sym_2(\bZ_v), 
   &&\bbc\in \bZ^\times_v, 
   && 4\det\bS=\mr{disc}(\cK_v/\bQ_v)
\end{align*}
\item[--] a continuous character $\bUps:\cK^\times\backslash\bA^\times_{\cK,f}\ra \Lambda^\times_{\GL(2)}$ extending $\omega_{\sC_1}$,
\item[--] a continuous character $\bLam:\cK^\times\backslash\bA^\times_{\cK,f}\ra \Lambda^\times_{\GSp(4)}$ extending $\omega_{\sC_2}$.
\end{enumerate}

\subsubsection{Notation for some $p$-adic characters}\label{sec:pchar}
We identify  $T^1_{\GL(2)}(\bZ_p), T^1_{\GSp(4)}(\bZ_p)$ with $\bZ^\times_p,\bZ^\times_p\times\bZ^\times_p$ via
\begin{align*}
   a&\longmapsto \diag(a,a^{-1}),
   &(a_1,a_2)&\longmapsto \diag(a_1,a_2,a^{-1}_1,a^{-1}_2).
\end{align*}
We denote by $\tau$ (resp. $(\tau_1,\tau_2)$) a continuous character $T^1_{\GL(2)}(\bZ_p)\ra\bar{\bQ}^\times_p$ (resp. $T^1_{\GSp(4)}(\bZ_p)\ra\bar{\bQ}^\times_p$). When the characters are arithmetic, {\it i.e.}
\begin{align*}
   \tau(x)&=x^l\xi(x),
   &\tau_i(x)&=x^{l_i}\xi_i(x),\quad i=1,2
\end{align*}
for some integers $l,l_1,l_2$ and finite order characters $\xi,\xi_1,\xi_2$  of $\bZ^\times_p$, we write
\begin{align*}
   \tau&=(l,\xi),
   &(\tau_1,\tau_2)=(l_1,l_2,\xi_1,\xi_2),
\end{align*}
and call $l$ (resp. $(l_1,l_2)$) the algebraic part of $\tau$ (resp. $(\tau_1,\tau_2)$), and $\xi$ (resp. $(\xi_1,\xi_2)$) the finite order part. We view $\tau,\tau_1,\tau_2$ also as a $\bar{\bQ}_p$-valued character of $\bQ^\times\backslash \bA^\times_{\bQ,f}/\hat{\bZ}^{\times,p}$ via the isomorphism $\bZ^\times_p\stackrel{\sim}{\ra} \bQ^\times\backslash \bA^\times_{\bQ,f}/\hat{\bZ}^{\times,p}$ (induced by the embedding $\bZ^\times_p\hra \bA^\times_{\bA,f}$). 

We denote by $\kappa$ a continuous character $\bQ^\times\backslash\bA^\times_{\bQ,f}/U^p\ra \ra\bar{\bQ}^\times_p$, and when it is arithmetic, we write 
\[
    \kappa=(k,\chi)
\]
with $k$ the algebraic part and $\chi$ the finite order part.\\

In this paper, we call an arithmetic tuple $(\kappa,\tau,\tau_1,\tau_2)$ classical if their algebraic parts satisfy
\begin{equation}\label{eq:k-crt} 
  -\min\{l_1+l_2,l+l_2\}+2\leq k\leq -\max\{l_1,l\}-1.    
\end{equation}
This condition corresponds to 
\[
  -\frac{\min\{-l_1+l_2+l,\,l_1+l_2-l\}}{2}+2 \leq k+\frac{l_1+l_2+l}{2}\leq \frac{\min\{-l_1+l_2+l,\,l_1+l_2-l\}}{2}-1
\]
equivalent to $s=k+\frac{l_1+l_2+l}{2}$ being a critical point for the degree-8 $L$-function $L(s,\Pi\times\pi)$ with $\Pi_\infty$ (resp. $\pi_\infty$) isomorphic to the holomorphic discrete series of weight $(l_1,l_2)$ (resp. $l$).

\subsubsection{Convention on Nebentypus at $p$ and central characters}
Given arithmetic $\tau=(l,\xi),(\tau_1,\tau_2)=(l_1,l_2,\xi_1,\xi_2)$, we use the convention that the classical automorphic forms on $\GL(2),\GSp(4)$ contained in $V_{\GL(2)}[\tau],V_{\GSp(4)}[\tau_1,\tau_2]$ have 
\begin{enumerate}
\item[--] weight $l,(l_1,l_2)$,
\item[--]  nebetypus such that the right translation of $B_{\GL(2)}(\bZ_p),B_{\GSp(4)}(\bZ_p)$ act by the character
\begin{align}
   \label{eq:pneben}\begin{bmatrix}a&\ast\\ &d\end{bmatrix}
   &\longmapsto
   \xi(a),
   &\begin{bmatrix} a_1&\ast&\ast&\ast\\&a_2&\ast&\ast\\ && a^{-1}_1\nu\\ &&\ast& a^{-1}_2\nu\end{bmatrix}
   &\longmapsto
   \xi_1(a_1)\xi_2(a_2),
\end{align}
\item[--] central character equal to the product of a finite order character unramified away from $p$ and the classical character corresponding to 
\begin{align*}
   &\tau:\bQ^\times\backslash \bA^\times_{\bQ,f}/\hat{\bZ}^{\times,p}\lra  \bar{\bQ}^\times_p,
   &&\tau_1\tau_2:\bQ^\times\backslash \bA^\times_{\bQ,f}/\hat{\bZ}^{\times,p}\lra  \bar{\bQ}^\times_p.
\end{align*}
\end{enumerate} 
(Note that with this convention, the central characters of classical cuspidal automorphic forms in $V_{\GL(2)}[\tau],V_{\GSp(4)}[\tau_1,\tau_2]$ are not necessarily unitary.)

\subsection{The Eisenstein measure}
\subsubsection{The Siegel Eisenstein series and its Fourier coefficients}
For a classical tuple $(\kappa,\tau,\tau_1,\tau_2)$ with $\kappa$ fixed by $U^p$, put
\begin{align*}
   E^\Sieg_{\kappa,\tau,\tau_1,\tau_2}(g)= &\, |\nu_{\GU(3,3)}(g)|^{\frac{3}{2}(l_1+l_2+l)}
   \cdot \frac{\prod_{j=0}^2\Gamma\left(s+\left|k+\frac{l+l_1+l_2}{2}-\frac{1}{2}\right|+3-j\right)}{(-2i)^{\left|k+\frac{l+l_1+l_2}{2}-\frac{1}{2}\right|+\frac{3}{2}}\,\pi^{3\left(s+\left|k+\frac{l+l_1+l_2}{2}-\frac{1}{2}\right|+2\right)}}\\
    &\times d^S_3\big(s,\Lambda_0\Upsilon_0(\chi\circ\Nm)\big) \cdot
    E^{\Sieg}\left(g;\dsec_{l,\,l_1,l_2,\xi,\,\xi_1,\xi_2}(s,\chi,\Lambda_0\Upsilon_0)\right)\Big|_{s=k+\frac{l_1+l_2+l}{1}-\frac{1}{2}}
\end{align*}
where $\Lambda_0=\Lambda\,|\cdot|^{-\frac{l_1+l_2}{2}},\Upsilon_0=\Upsilon \,|\cdot|^{-\frac{l}{2}}$ with $\Lambda,\Upsilon$ the classical characters corresponding to the specializations at $\tau_1,\tau_2,\tau$ of $\bLam,\bUps$ fixed in \S\ref{sec:setup}, and  the section 
\[
    \dsec_{l,\,l_1,l_2,\xi,\,\xi_1,\xi_2}(s,\chi,\Lambda_0\Upsilon_0)\in\bigotimes_v I_v(s,\chi,\Lambda_0\Upsilon_0)
\]
is chosen as in \cite[\S3.3]{pFL}. (For the place $p$, the $\eta^\circ_{\sPi_1,p},\eta^\circ_{\sPi_2,p},\eta^\circ_{\sPi_3,p}$ in {\it loc.cit} correspond to $\xi^{-1}_2,\xi^{-1}_1,\xi^{-1}_1\xi^{-1}_2$ here, and $\eta^\circ_{\pi_p,1},\eta^\circ_{\pi_p,2}$ in {\it loc.cit} correspond to $\xi^{-1},\triv$ here.)\\

Given $\bbeta\in\Her_3(\cK)$ and a nearly holomorphic automorphic form $F$ on $\GU(3,3)$, we define the $\bbeta$-th polynomial Fourier coefficient of $F$ at $g_f\in \GU(3,3)(\bA_{\bQ,f})$ as the polynomial $F_{\bbeta}(g)\in \bC[Y]$, where $\bC[Y]$ denotes the ring of polynomials in entries of $Y=(Y_{ij})_{1\leq i,l\leq 3}$, such that for $z\in \left\{z\in M_{3,3}(\bC)\mid i(\ltrans{\bar{z}}-z)>0\right\}$,
\begin{align*}
	&F_{\bbeta}(g_f)\left(Y=\left(\frac{z-\ltrans{\bar{z}}}{2i}\right)^{-1}\right)\cdot e^{2\pi i\Tr z\bbeta}\\
	=&\,\int_{\Her_3(\cK)\backslash\Her_3(\bA_\cK)} F\left(\begin{bmatrix}\bid_3&\varsigma\\0&\bid_3\end{bmatrix}g_f\begin{bsm}\left(\frac{z-\ltrans{\bar{z}}}{2i}\right)^{1/2}&\frac{z+\ltrans{\bar{z}}}{2}\left(\frac{z-\ltrans{\bar{z}}}{2i}\right)^{-1/2}\\0&\left(\frac{z-\ltrans{\bar{z}}}{2i}\right)^{-1/2}\end{bsm}_\infty\right)\, \psi_{\bA_\bQ}\left(-\Tr\bbeta\fx\right)\,d\fx.
\end{align*} 
(The existence of such a polynomial $F_{\bbeta}(g)$ is implied by the definition of nearly holomorphic modular forms.) Similarly, one defines polynomial Fourier coefficients of nearly holomorphic forms on $\GU(1,1),\GL(2)$ (resp. $\GSp(4)$) indexed by $\bQ=\Her_1(\cK)$ (resp. $\Sym_2(\bQ)$).  (See \cite[\S13.11]{Sh00} or  \cite[\S3.2]{pFL} for the definition of nearly holomorphic forms and \cite[\S2.4]{NHF} for their interpretations as global sections of automorphic sheaves over Shimura varieties.)

By our choice of the archimedean component of $ \dsec_{l,\,l_1,l_2,\xi,\,\xi_1,\xi_2}(s,\chi,\Lambda_0\Upsilon_0)$, the Eisenstein series $E^\Sieg_{\kappa,\tau,\tau_1,\tau_2}$ is a nearly holomorphic automorphic form on $\GU(3,3)$. We have the following proposition on its polynomial Fourier coefficients. 

\begin{prop}
For $\fA\in\GL(2,\bA^p_{\cK,f})$, $\nu\in \bA^{\times,p}_f$ and $\bbeta\in\Her_3(\cK)$, $\left(E^\Sieg_{\kappa,\tau,\tau_1,\tau_2}\right)_{\bbeta}=0 $ unless $\bbeta>0$, and for $\bbeta>0$,
\begin{align*}
   &\left(E^\Sieg_{\kappa,\tau,\tau_1,\tau_2}\right)_{\bbeta}\left(\begin{bmatrix}\fA\\&\nu\ltrans{\bar{\fA}}^{-1}\end{bmatrix}^p_f\right)(Y=0)
   = A(\bbeta;\kappa,\tau,\tau_1,\tau_2) \left(\begin{bmatrix}\fA\\&\nu\ltrans{\bar{\fA}}^{-1}\end{bmatrix}^p_f\right)
   \cdot C_{k,l,l_1,l_2}(\bbeta)
\end{align*}
with
\begin{equation}\label{eq:Abeta}
\begin{aligned}
    &A(\bbeta;\kappa,\tau,\tau_1,\tau_2) \left(\begin{bmatrix}\fA\\&\nu\ltrans{\bar{\fA}}^{-1}\end{bmatrix}^p_f\right)\\
    =&\,\Lambda\Upsilon(\det(\nu\ltrans{\bar{\fA}}^{-1}) \cdot  (\chi|\cdot|^k)(\det(\nu\fA^{-1}\ltrans{\bar{\fA}}^{-1})) \\
   &\times\prod_{v\nmid Np\infty} h_{v,\nu^{-1}\ltrans{\bar{\fA}}\bbeta\fA}\left(\Lambda_{\bQ,v}\Upsilon_{\bQ,v}\chi^2_v(\varpi_v)|\varpi_v|^{2k+2}\right)
    \prod_{v\mid N}\cF\Schw^\vol_{v,\val_v(N)} \left(\nu^{-1}_v\ltrans{\bar{\fA}}_v\bbeta\fA_v\right)\\
    &\times \Lambda^{\circ-1}_{\bar{\fp}}\xi_1(\varrho_\fp(\beta_{13}))\varrho_\fp(\beta_{13})^{-r_{\sLambda,2}+l_1} \cdot \Lambda^{\circ-1}_{\fp}\xi_1(\varrho_{\bar{\fp}}(\beta_{13}))\varrho_{\bar{\fp}}(\beta_{13})^{-r_{\sLambda,1}+l_1}\\
   &\times \xi_2\xi\chi^\circ_p \left(\frac{\bar{\beta}_{13}\beta_{23}-\beta_{13}\bar{\beta}_{23}}{\alphaS-\alphabS}\right)\left(\frac{\bar{\beta}_{13}\beta_{23}-\beta_{13}\bar{\beta}_{23}}{\alphaS-\alphabS}\right)^{l_2+l+k-2}\\
   &\times \xi_1\xi_2\chi^\circ_p\left( \frac{(\alphaS-\alphabS)(\beta_{12}-\bar{\beta}_{12})}{2}\right)\left( \frac{(\alphaS-\alphabS)(\beta_{12}-\bar{\beta}_{12})}{2}\right)^{l_1+l_2+k-2},
\end{aligned}
\end{equation}
and
\begin{equation}\label{eq:Cbeta}
     C_{k,l,l_1,l_2}(\bbeta)
   =\left\{\begin{array}{ll}\left(\left(\frac{\beta_{12}-\bar{\beta}_{12}}{2}\det\begin{bmatrix}\bar{\beta}_{13}&\bar{\beta}_{23}\\ \beta_{13}&\beta_{23}\end{bmatrix}\right)^{-1}\det\bbeta)\right)^{2k+l_1+l_2+l-1}, &k+\frac{l_1+l_2+l}{2}\geq \frac{1}{2},\\[4ex]
   1,& k+\frac{l_1+l_2+l}{2}\leq \frac{1}{2}.
   \end{array}\right.
\end{equation}
Here, $\Lambda,\Upsilon$ are the classical characters associated to the specialization of $\bLam$ at $(\tau_1,\tau_2)$, $\bUps$ at $\tau$, and $(r_{\sLambda,1},r_{\sLambda,2})$ is the $\infty$-type of $\Lambda$, $\varrho_p,\varrho_{\bar{\fp}}$ are embedding of $\cK$ into $\bar{\bQ}_p$ inducing $\fp,\bar{\fp}$. (See \eqref{eq:char-circ} for the notation superscript $^\circ$ on characters.)
\end{prop}

This proposition follows immediately from Proposition~\cite[Proposition~3.5.1]{pFL}. (The characters $\eta^\circ_{\sPi_1,p},\eta^\circ_{\sPi_2,p},\eta^\circ_{\sPi_3,p}$ in {\it loc.cit} correspond to $\xi^{-1}_2,\xi^{-1}_1,\xi^{-1}_1\xi^{-1}_2$ here,  $\eta^\circ_{\pi_p,1},\eta^\circ_{\pi_p,2}$ in {\it loc.cit} correspond to $\xi^{-1},\triv$ here, $k+\frac{\epsilon}{2},\Lambda,\Xi$ in {\it loc.cit} corresponds to $k+\frac{l+l_1+l_2}{2},\Lambda^{-c}_0,\Lambda_0\Upsilon_0$ here.)

\subsubsection{The $p$-adic measure interpolating restrictions of Siegel Eisenstein series}
\begin{thm}\label{thm:EisF}
There exists a $p$-adic measure
\[
    \mu_\cE\in\Meas\left(\bQ^\times\backslash\bA^\times_{\bQ,f}/U^p,\cM_{\GL(2),\ord}\wh{\otimes}_{\cO} \cM_{\GSp(4),\ord}\right)
\]
such that for all classical $(\kappa,\tau,\tau_1,\tau_2)=((k,\chi),(l,\xi),(l_1,l_2,\xi_1,\xi_2))$ with $k,l,l_1,l_2$ satisfying \eqref{eq:k-crt}, denoting by $\spe_{\tau,(\tau_1,\tau_2)}$ the specialization map at $\tau,(\tau_1,\tau_2)$, we have
\begin{align*}
   &\spe_{\tau,(\tau_1,\tau_2)}\left(\mu_\cE(\kappa)\right)\\
   =&\,e^{\GL(2)}_\ord e^{\GSp(4)}_\ord \,{\rm Proj}_{K^p_{\GL(2)},K^p_{\GSp(4)}}\left(\left.E^\Sieg_{\kappa,\tau,\tau_1,\tau_2}\right|_{\GL(2)\times\GSp(4)}\right)\cdot \spe_{\tau,(\tau_1,\tau_2)}(\omega^{-1}_{\sC_1}\omega^{-1}_{\sC_2}\circ{\rm det}_{\GL(2)}).
\end{align*}
Here, $-\,|_{\GL(2)\times\GSp(4)}$ means restriction to $\GL(2)\times_{\GL(1)}\GSp(4)$ followed by extension by zero to $\GL(2)\times\GSp(4)$, and ${\rm Proj}_{K^p_{\GL(2)},K^p_{\GSp(4)}}$ is the projection $\int_{K^p_{\GL(2)},K^p_{\GSp(4)}} \text{translation by }(h,g)\,dh dg$.

\end{thm}

\begin{proof}
For $G=\GSp(4),\GL(2)$, denote by $\Meas\left(T^1_G(\bZ_p),V_{G,\ord}\right)$ the set of $p$-adic measures on $T^1_G(\bZ_p)$ valued in $V_{G,\ord}$. The group $T^1_G(\bZ_p)$ acts on it in two ways: via its action by translation on $T^1_G(\bZ_p)$  and its action on $V_{G,\ord}$. Let
\[
   \Meas\left(T^1_G(\bZ_p),V_{G,\ord}\right)^\natural\subset \Meas\left(T^1_G(\bZ_p),V_{G,\ord}\right)   
\]
be the subset on which the two action of $T^1_G(\bZ_p)$ agree and make it a  $\tilde{\Lambda}_G$-module. Unfolding the definition of $\cM_{G,\ord}$ gives  a natural map from $ \Meas\left(T^1_G(\bZ_p),V_{G,\ord}\right)^\natural$ to $\cM_{G,\ord}$ such that for each $p$-adic weight $\underline{\tau}\in \Hom_{\cont}\left(T^1_G(\bZ_p),\bar{\bQ}^\times_p\right)$, we have the commutative diagram
\begin{equation}
\begin{tikzcd}
    \Meas\left(T^1_G(\bZ_p),V_{G,\ord}\right)^\natural \arrow[r] \arrow[d,"\text{evaluate at $\utau$}"']& \cM_{G,\ord} \arrow[d]\\
    V_{G,\ord}[\utau] \arrow[r,"\eqref{eq:embd}"] &\cM_{G,\ord}\otimes_{\tilde{\Lambda}_G}\tilde{\Lambda}_G/\cP_{\utau}
\end{tikzcd}
\end{equation}
(cf. \cite[\S6.1.4]{SLF}).  Therefore, we can identify $\Meas\left(\bQ^\times\backslash\bA^\times_{\bQ,f}/U^p,\cM_{\GL(2),\ord}\wh{\otimes}_{\cO} \cM_{\GSp(4),\ord}\right)$ with 
\[
   \Meas\left((\bQ^\times\backslash\bA^\times_{\bQ,f}/U^p)\times T^1_{\GU(1,1),\ord}\times T^1_{\GSp(4),\ord},V_{\GU(1,1),\ord}\otimes_\cO V_{\GSp(4),\ord}\right)^\natural,
\]   
and we only need to show that there exists a $p$-adic measure $\mu'_\cE$ inside this space such that
\begin{equation}\label{eq:mu'}
\begin{aligned}
\mu'_\cE(\kappa,\tau,(\tau_1,\tau_2))
   =&\,e^{\GL(2)}_\ord e^{\GSp(4)}_\ord \,{\rm Proj}_{K^p_{\GL(2)},K^p_{\GSp(4)}}\left(\left.E^\Sieg_{\kappa,\tau,\tau_1,\tau_2}\right|_{\GL(2)\times\GSp(4)}\right)\\
    &\times \spe_{\tau,(\tau_1,\tau_2)}(\omega^{-1}_{\sC_1}\omega^{-1}_{\sC_2}\circ{\rm det}_{\GL(2)})
\end{aligned}
\end{equation}
for all classical $(\kappa,\tau,\tau_1,\tau_2)=((k,\chi),(l,\xi),(l_1,l_2,\xi_1,\xi_2))$. The existence of $\mu'_\cE$ can be shown by $p$-adic interpolation of Fourier coefficients.\\

Take sufficiently small tame level group $K^{p,\prime}_{\GL(2)}\subset K^p_{\GL(2)},K^{p,\prime}_{\GSp(4)}\subset K^p_{\GSp(4)}$ such that $\left.E^\Sieg_{\kappa,\tau,\tau_1,\tau_2}\right|_{\GL(2)\times\GSp(4)}$ is fixed by $K^{p,\prime}_{\GL(2)},K^{p,\prime}_{\GSp(4)}$ for all classical $(\kappa,\tau,\tau_1,\tau_2)$. Denote by $V'_{\GL(2)},V'_{\GSp(4)}$ the space of $p$-adic forms of tame levels $K^{p,\prime}_{\GL(2)},K^{p,\prime}_{\GSp(4)}$. Thanks to the strong approximation for $\Sp(4)$ and $\SL(2)$,  we can pick $\nu_i\in\bA^{\times,p}_{\bQ,f}$, $i=1,\dots,c_1$, and $\nu'_j\in\bA^{\times,p}_{\bQ,f}$, $j=1,\dots,c_2$, such that each connected component of the Shimura variety for $\GL(2)\times\GSp(4)$ of level $K^{p,\prime}_{\GL(2)}\GL(2)(\bZ_p)\times K^{p,\prime}_{\GSp(4)}\GSp(4,\bZ_p)$ contains a cusp corresponding to $\left(\begin{bsm}\bid_2\\&\nu_i\cdot\bid_2\end{bsm},\begin{bsm}1\\&\nu'_j\end{bsm}\right)$ for some $i,j$. Taking the Fourier coefficients at these cusps gives an injection
\begin{equation}\label{eq:q-exp}
   \varepsilon_{\qexp}:V'_{\GL(2)}\hat{\otimes}_{\cO} V'_{\GSp(4)}
   \lhra \cO\llb \bQ_{>0}\times \Sym_2(\bQ)_{> 0}\rrb^{\oplus c_1c_2},
\end{equation}
and the image of $\varepsilon_{\qexp}$ is closed in $\cO\llb\Sym_2(\bQ)_{> 0}\times \bQ_{> 0}\rrb^{\oplus c_1c_2}$ (for the $p$-adic topology).  (The injectivity and the closedness of the image follows from the irreducibiliy of Igusa towers \cite[Corollary~8.17]{HidaPAF}.)  We view $V'_{\GU(1,1)}\hat{\otimes}_{\cO} V'_{\GSp(4)}$ as a subspace of $\cO\llb\Sym_2(\bQ)_{>0}\times \bQ_{> 0}\rrb^{\oplus c_1c_2}$. The $\bU_p$-operators on $V'_{\GL(2)}\hat{\otimes}_{\cO} V'_{\GSp(4)}$  extend to operators on $\cO\llb\Sym_2(\bQ)_{> 0}\times \bQ_{> 0}\rrb^{\oplus c_1c_2}$. Writing elements in $\cO\llb\Sym_2(\bQ)_{> 0}\times \bQ_{> 0}\rrb^{\oplus c_1c_2}$  as $\sum_{i,j,\beta_1,\beta_2} a_{(i,j)}(\beta_1,\beta_2) \,q^{(\beta_1,\beta_2)}_{(i,j)}$, with summation over $1\leq i\leq c_1,1\leq j\leq c_2$, $\beta_1,\beta_2\in \Sym_2(\bQ)_{>0}\times \bQ_{> 0}$, the extension has the formula
\begin{align*}
   &U^{\GL(2)}_{p,m_3,0}U^{\GSp(4)}_{p,m_1,m_2,0}\left(\sum_{i,j,\beta_1,\beta_2} a_{i,j}(\beta_1,\beta_2)\,q^{(\beta_1,\beta_2)}\right)\\
   =&\, \sum_{x\in\bZ/p^{m_1-m_2}\bZ}
   \sum_{i,j,\beta_1,\beta_2} a_{(i,j)}\left(\begin{bsm}p^{m_1-m_2}&\\ x&1\end{bsm}p^{m_2}\beta_1\begin{bsm}p^{m_1-m_2}&x\\&1\end{bsm},p^{2m_3}\beta_2\right)\,q^{(\beta_1,\beta_2)}.
\end{align*}

On the other hand, letting $\cN_{G}$ denote the space of classical nearly holomorphic forms on $G=\GSp(4),\GL(2)$ over $\cO$ of tame level $K^{p,\prime}_G$, all levels at $p$ containing $U_G(\bZ_p)$ and all weights,  vanishing along all $p$-adic cusps, the unit root splitting \cite[Theorem~4.1]{TDwork} gives rise to a map
\[
    \imath_{p\adic}: \cN_{\GL(2)}\otimes_\cO\cN_{\GSp(4)}\lra V'_{\GL(2)}\hat{\otimes}_{\cO} V'_{\GSp(4)},
\] 
which is actually an embedding by \cite[Proposition~3.12.1]{NHF}. Moreover, for a nearly holomorphic form $F$ on $\GL(2)\times\GSp(4)$, we have
\begin{align*}
   &(\beta_1,\beta_2)\text{-coefficient of }(i,j)\text{-component of }\varepsilon_{\qexp}\left(\imath_{p\adic} \left(F\right)\right)\\
   =&\, F_{\beta_1,\beta_2}\left(\begin{bsm}\bid_2\\&\nu_i\cdot\bid_2\end{bsm},\begin{bsm}1\\&\nu'_j\end{bsm}\right)(Y=0).
\end{align*}
It follows that,  
\begin{align*}
    &\varepsilon_{\qexp}\left(\imath_{p\adic} \left(E^\Sieg_{\kappa,\tau,\tau_1,\tau_2}\big|_{\GL(2)\times\GSp(4)}\right)\right)\\
    = &\sum_{\substack{(i,j),(\beta_1,\beta_2)\\ \nu_i=\nu'_j}} \left(\sum_{\bbeta\in\Her_3(\cK)_{>0},\,\frac{\bbeta+\bar{\bbeta}}{2}=\begin{bsm}\beta_1&\ast\\ \ast&\beta_2\end{bsm}}
     A(\bbeta;\kappa,\tau,\tau_1,\tau_2) \left(\begin{bmatrix}\bid_3\\&\nu_{i}\bid_3\end{bmatrix}^p_f\right) C_{k,l,l_1,l_2}(\bbeta) \right) q^{(\beta_1,\beta_2)}_{(i,j)}
\end{align*}
with $A(\bbeta;\kappa,\tau,\tau_1,\tau_2),C_{k,l,l_1,l_2}(\bbeta)$ given by the formulas \eqref{eq:Abeta} and \eqref{eq:Cbeta}. From \eqref{eq:Abeta}, it is easy to see that for each $\bbeta$ and $i,j$ such that $\nu_i=\nu'_j$, there exists a $p$-adic measure
\[
   \mu_{(i,j),\bbeta}\in \Meas\left(\bQ^\times\backslash\bA^\times_{\bQ,f}/U^p\times T^1_{\GL(2)}(\bZ_p)\times T^1_{\GSp(4)}(\bZ_p),\cO\right)
\]
such that for call classical tuples $(\kappa,\tau,\tau_1,\tau_2)$,
\[
    \mu_{(i,j),\bbeta}(\kappa,\tau,\tau_1,\tau_2)=A(\bbeta;\kappa,\tau,\tau_1,\tau_2) \left(\begin{bmatrix}\bid_3\\&\nu_i\bid_3\end{bmatrix}^p_f\right).
\]
Define $\mu_{\cE,\qexp}\in \Meas\left(\bQ^\times\backslash\bA^\times_{\bQ,f}/U^p\times T^1_{\GL(2)}(\bZ_p)\times T^1_{\GSp(4)}(\bZ_p),\cO\llb\bQ_{>0}\times \Sym_2(\bQ)_{> 0}\rrb^{\oplus c_1c_2}\right)$ as
\[
   \mu_{\cE,\qexp}=\sum_{\substack{(i,j),(\beta_1,\beta_2)\\ \nu_i=\nu'_j}} \left(\sum_{\bbeta\in\Her_3(\cK)_{>0},\,\frac{\bbeta+\bar{\bbeta}}{2}=\begin{bsm}\beta_1&\ast\\ \ast&\beta_2\end{bsm}}
     \mu_{(i,j),\bbeta}\right) q^{(\beta_1,\beta_2)}_{(i,j)}.  
\] 
Observe that $C_{k,l,l_1,l_2}(\beta)\equiv 1\mod p^m$ for all $\bbeta$ such that $\bbeta+\bar{\bbeta}\equiv 0\mod p^m$. Thus, for all classical $(\kappa,\tau,\tau_1,\tau_2)$,
\begin{align*}
    &\left(U^{\GL(2)}_{p,1,0} U^{\GSp(4)}_{p,2,1,0}\right)^m \mu_{\cE,\qexp}(\kappa,\tau,\tau_1,\tau_2) \\
    \equiv &\,\varepsilon_{\qexp}\left( \left(U^{\GL(2)}_{p,1,0} U^{\GSp(4)}_{p,2,1,0}\right)^m\imath_{p\adic} \left(E^\Sieg_{\kappa,\tau,\tau_1,\tau_2}\big|_{\GL(2)\times\GSp(4)}\right)\right) \mod p^m.
\end{align*}
We deduce that the limit
\[
    \lim_{n\to \infty}\left(U^{\GL(2)}_{p,1,0} U^{\GSp(4)}_{p,2,1,0}\right)^{n!} \mu_{\cE,\qexp}
\]
exists (because the limit of the right hand side exists and the classical points are dense), and interpolates $\varepsilon_{\qexp} \left(e^{\GL(2)}_\ord e^{\GSp(4)}_\ord\imath_{p\adic} \left(E^\Sieg_{\kappa,\tau,\tau_1,\tau_2}\big|_{\GU(1,1)\times\GSp(4)}\right)\right)$ at all classical $(\kappa,\tau,\tau_1,\tau_2)$. Denote this limit by 
\[
   e^{\GL(2)}_\ord e^{\GSp(4)}_\ord\left(\mu_{\cE,\qexp}\right)\in \Meas\left(\bQ^\times\backslash\bA^\times_{\bQ,f}/U^p\times T^1_{\GL(2)}(\bZ_p)\times T^1_{\GSp(4)}(\bZ_p),\cO\llb\bQ_{>0}\times \Sym_2(\bQ)_{> 0}\rrb^{\oplus c_1c_2}\right).
\] 
Since the classical points are dense in the weight space and the image of \eqref{eq:q-exp} is dense,  this limit must come from the $q$-expansion of a $p$-adic measure valued in $p$-adic forms, {\it i.e.}  there exists 
\[
   \mu'_\cE\in \Meas\left((\bQ^\times\backslash\bA^\times_{\bQ,f}/U^p)\times T^1_{\GU(1,1),\ord}\times T^1_{\GSp(4),\ord},V'_{\GU(1,1),\ord}\otimes_\cO V'_{\GSp(4),\ord}\right)^\natural,
\]
such that
\[
   \varepsilon_{\qexp}(\mu'_\cE)=e^{\GL(2)}_\ord e^{\GSp(4)}_\ord\left(\mu_{\cE,\qexp}\right).
\]
Then ${\rm Proj}_{K^p_{\GL(2)},K^p_{\GSp(4)}}(\mu'_\cE)\cdot \omega^{-1}_{\sC_1}\omega^{-1}_{\sC_2}\circ{\rm det}_{\GU(1,1)}$ satisfies \eqref{eq:mu'}.
\end{proof}

\subsection{The four-variable $p$-adic $L$-function and its interpolation formula I}

\subsubsection{Hida families and idempotent operators}

Let $F_{\sC_1},F_{\sC_2}$ be the function fields of the irreducible components $\sC_1,\sC_2$ fixed in \S\ref{sec:setup}. Then the maps $\tilde{\Lambda}_{\GL(2)}\ra F_{\sC_1}, \tilde{\Lambda}_{\GSp(4)}\ra F_{\sC_2}$ factors through projections $\tilde{\Lambda}_{\GL(2)}\ra\Lambda_{\GL(2)}$, $\tilde{\Lambda}_{\GSp(4)}\ra\Lambda_{\GSp(4)}$ induced by characters of $T^1_{\GL(2)}(\bZ/p\bZ)$ and $T^1_{\GSp(4)}(\bZ/p\bZ)$. We view $F_{\sC_1}$ as an algebra over $\Lambda_{\GL(2)}$ and $F_{\sC_2}$ as an $\Lambda_{\GSp(4)}$-algebra through these factorizations.

Denote by $\bI_{\sC_1},\bI_{\sC_2}$ the integral closures of $\Lambda_{\GL(2)},\Lambda_{\GSp(4)}$ inside  $F_{\sC_1},F_{\sC_2}$. The universal ordinary Hecke algebras $\bT_{\GL(2),\ord},\bT_{\GSp(4),\ord}$ are known to be reduced. Therefore, we have
\begin{align*}
   \bT_{\GL(2),\ord}\otimes F_{\sC_1} &= F_{\sC_1}\oplus R_{\sC_1},
   &\bT_{\GSp(4),\ord}\otimes F_{\sC_2} & =F_{\sC_2}\oplus R_{\sC_2}
\end{align*}
as  $F_{\sC_1}$-algebras and $F_{\sC_2}$-algebras such that the projection onto the first factor agrees with the natural maps $\bT_{\GL(2),\ord}\ra \bI_{\sC_1}$, $\bT_{\GSp(4),\ord}\ra \bI_{\sC_2}$. Let
\begin{align}
   \label{eq:HeckeProj}\mathds{1}_{\sC_1}&\in \bT_{\GL(2),\ord}\otimes F_{\sC_1},
   &\mathds{1}_{\sC_2}&\in \bT_{\GSp(4),\ord}\otimes F_{\sC_2}
\end{align}
be the idempotent associated to the first factor in the above decomposition. 

\subsubsection{The modified Euler factors at $p$ and $\infty$}\label{sec:modEuler}
We let 
\begin{align*}
   \lambda_{\GL(2),0,1}&\in\bI^\times_{\sC_1},
   &&\lambda_{\GSp(4),0,0,1}, \lambda_{\GSp(4),1,0,0}\in \bI^\times_{\sC_2}
\end{align*}
denote the eigenvalues of the $\bU_p$-operators associated to $\begin{bsm}p\\&1\end{bsm},\begin{bsm}p\\&p\\&&1\\&&&1\end{bsm}, \begin{bsm}p\\&1\\&&p^{-1}\\&&&1\end{bsm}$.

Given a point $x\in \sC_1(\bar{\bQ}_p)\times\sC_2(\bar{\bQ}_p)$ where the weight projection map  $\Lambda_{\GL(2)}\hat{\otimes}_\cO \Lambda_{\GSp(4)}\ra \bT_{\GL(2),\ord}\hat{\otimes}_\cO\bT_{\GSp(4),\ord}$ is \'{e}tale and the image of $x$ is an arithmetic tuple $(\tau,\tau_1,\tau_2)=(l,\xi,l_1,\xi_1,l_2,\xi_2)$, we let \begin{align*}
    &\sS_{\GL(2),x} &&\text{(resp.  $\sS_{\GSp(4),x}$)}
\end{align*}
be an orthogonal basis of the space spanned by ordinary cuspidal holomorphic forms on $\GL(2)$ of weight $l$, tame level $K^p_{\GL(2)}$ (resp. $\GSp(4)$ of weight $(l_1,l_2)$, tame level $K^p_{\GSp(4)}$) and nebentypus at $p$ given by \eqref{eq:pneben}, belonging to the Hecke eigenspace parameterized by $x$. Let $\pi_x$ be the unitary irreducible automorphic representation of $\GL(2,\bA_\bQ)$ generated by forms $\sS_{\GL(2),x}$ twisted by a real power of $|\det|$, and $\Pi_x$ be a unitary irreducible automorphic representation of $\GSp(4,\bA_\bQ)$ inside the representation generated by forms in $\sS_{\GSp(4),x}$ twisted by a real power of $|\det|$. (There can be more than one choices of $\Pi_x$, but the partial $L$-function and modified Euler factors at $p,\infty$ do not depend on the choice of $\Pi_x$.)  Let $L^S(s,\Pi_x\times\pi_x\times\chi)$ be the degree 8 partial L-function, and 
\begin{equation} \label{eq:Ep}
    E_p(s,\Pi_x\times\pi_x\times\chi)=\gamma_p\left(s,\chi_p\eta_{x,1}\eta'_{x,1}\right)^{-1} \gamma_p\left(s,\chi_p\eta_{x,1}\eta^{\prime}_{x,2}\right)^{-1}  
    \gamma_p\left(s,\pi_{x,p}\times\chi_p\eta'_{x,3}\right)^{-1}
\end{equation}
where the characters $\eta_{x,1},\eta_{x,2},\eta'_{x,1},\eta'_{x,2},\eta'_{x,3}$ are:

\begin{align*}
  \eta_{x,1}(a)&=\xi(a|a|_p) \left(p^{-\frac{l-1}{2}}\lambda_{\GL(2),0,1}(x)\right)^{\val_p(a)},
  &\eta_{x,2}(a)&=\left(p^{\frac{l-1}{2}}(\omega_{\sC_1,p}(p)\lambda^{-1}_{\GL(2),0,1})(x)\right)^{\val_p(a)},
\end{align*}
\begin{align*}
  \eta'_{x,1}(a)&=\xi_1(a|a|_p)\left(p^{\frac{-l_1+l_2-1}{2}} \omega_{\sC_2,p}(p)\lambda_{\GSp(4),1,0}\lambda_{\GSp(4),0,1}^{-1}(x)\right)^{\val_p(a)},\\
  \eta'_{x,2}(a)&=\xi_2(a|a|_p)\left(p^{\frac{l_1-l_2+1}{2}} (\lambda_{\GSp(4),0,1}\lambda_{\GSp(4),1,0}^{-1})(x)\right)^{\val_p(a)}\\
  \eta'_{x,3}(a)&=\xi_1\xi_2(a|a|_p)\left(p^{-\frac{l_1+l_2-3}{2}}\lambda_{\GSp(4),0,1}(x)\right)^{\val_p(a)}
\end{align*}
and
\begin{equation}\label{eq:Einf}
\begin{aligned}
    E_\infty(s,\Pi_x\times\pi_x\times\chi)=&\,e^{-(4s+l_1+l_2+l)\cdot \frac{\pi i}{2}}\Gamma_\bC\left(s+\frac{l_1+l_2+l}{2}-2\right)\Gamma_\bC\left(s+\frac{l_1+1_2-l}{2}-1\right)\\
   &\times\Gamma_\bC\left(s+\frac{-l_1+l_2+l}{2}-1\right)\Gamma_\bC\left(s+\frac{l_1-1_2+l}{2}\right),
\end{aligned}
\end{equation}
where $\Gamma_\bC(s)=2(2\pi)^{-s}\Gamma(s)$. (The factors $E_p,E_\infty$ are obtained by unfolding the definitions in \cite{CoPerrin,Coates}.)

\subsubsection{The modified Petersson inner product and Bessel period}\label{sec:modper}
For ordinary holomorphic automorphic forms $f\in\pi,\varphi\in\Pi$, we define the modified Petersson inner product $\bfP(f,f),\bfP(\varphi,\varphi)$ and the modified Bessel period $\bes^\dagger_{\bS,\Lambda}(\varphi)$ as follows:
\begin{equation}\label{eq:Pet2}
   \bfP(f,f)=\lambda_p(f)^{-m}\int_{[\GL(2)]} f(g)\,f\left(g\begin{bmatrix}&1\\1\end{bmatrix}_{p\infty} \begin{bmatrix}p^{m}\\&p^{-m}\end{bmatrix}_p\right)\cdot\omega_\pi(\det g)^{-1}\,dg,
\end{equation}
\begin{align}
    \bfP(\varphi,\varphi)=&\,\lambda_{p,1}(\varphi)^{-m_1}\lambda_{p,2}(\varphi)^{-m_2} \label{eq:Pet} \\
    &\times\int_{[\GSp(4)]} \varphi(g)\,\varphi\left(g\begin{bmatrix}&\bid_2\\\bid_2\end{bmatrix}_{p\infty}\begin{bsm}p^{m_1}\\&p^{m_2}\\&&p^{-m_1}\\&&&p^{-m_2}\end{bsm}_p\right)\cdot \omega_\Pi(\nu_g)^{-1} \,dg,\nonumber\\    
    \bes^\dagger_{\bS,\Lambda}(\varphi)=&\,
     \lambda_{p,1}(\varphi)^{-m_1}\lambda_{p,2}(\varphi)^{-m_2}
     \cdot \bes_{\bS,\Lambda}\left(\begin{bsm}p^{m_1}\\&p^{m_2}\\&&p^{-m_1}\\&&&p^{-m_2}\end{bsm}_p\varphi\right),\label{eq:bes-dagger}
\end{align}
with $m_1\gg m_2\gg 0$, $m\gg 0$, where $\lambda_p(f),\lambda_{p,1}(\varphi),\lambda_{p,2}(\varphi)$ are defined by
\begin{align*}
   \lambda_p(f)^m f&=\int_{U_{\GL(2)}(\bZ_p)}  \pi_p\left(u\begin{bmatrix}p^m\\&p^{-m}\end{bmatrix}\right)f\,du,\\
   \lambda_{p,1}(\varphi)^{m_1}\lambda_{p,2}(\varphi)^{m_2} \varphi&=\int_{U_{\GSp(4)}(\bZ_p)}  \Pi_p\left(u\begin{bmatrix}p^{m_1}\\&p^{m_2}\\&&p^{-m_1}\\&&&p^{-m_2}\end{bmatrix}\right)\varphi\,du.
\end{align*}
(One can check that the right hand sides of \eqref{eq:Pet2}\eqref{eq:Pet}\eqref{eq:bes-dagger} do not depend on $m_1,m_2,m$ as long as $m_1-m_2,m_2,m$ are sufficiently large. See \cite[Proposition~2.7.1]{pFL} for a proof of this for \eqref{eq:bes-dagger}.)

\subsubsection{The four-variable $p$-adic $L$-function}

With the various factors defined in \S\S\ref{sec:modEuler}, \ref{sec:modper}, we apply the idempotent in \eqref{eq:HeckeProj} to the $p$-adic measure $\mu_\cE$ constructed in Theorem~\ref{thm:EisF} to obtain the following theorem.

\begin{thm}\label{thm:pFL1}
Given the data in \S\ref{sec:setup}, there exists $p$-adic measure $\mu^S_{\sC_1,\sC_2}$ satisfying the interpolation properties: Suppose that $x\in \sC_1(\bar{\bQ}_p)\times\sC_2(\bar{\bQ}_p)$ is a point at which the weight projection map $\Lambda_{\GL(2)}\hat{\otimes}_\cO \Lambda_{\GSp(4)}\ra \bT_{\GL(2),\ord}\hat{\otimes}_\cO\bT_{\GSp(4),\ord}$ is \'{e}tale. Then $ \cL^S_{\sC_1,\sC_2}$ has no poles along $x$. Let $\tau\in\Hom_\cont\left(T^1_{\GL(2)}(\bZ_p),\bar{\bQ}_p\right)$, $(\tau_1,\tau_2)\in \Hom_\cont\left(T^1_{\GSp(4)}(\bZ_p),\bar{\bQ}_p\right)$ be the projection of $x$ to the weight space. For a character $\kappa\in \Hom_\cont\left(\bQ^\times\backslash\bA^\times_{\bQ,f}/U^p,\bar{\bQ}_p\right)$ such that $(\kappa,\tau,\tau_1,\tau_2)$ is classical (as defined at the end of \S\ref{sec:pchar}),
\begin{align*}
    \mu^S_{\sC_1,\sC_2}(\kappa,x)
    =&\,\sum_{f\in \sS_{\GL(2),x}} \frac{f_{\bbc}f}{\bfP(f,f)}\sum_{\varphi\in \sS_{\GSp(4)},x} \frac{\bes^\dagger_{\bS,\Lambda}\left(\varphi\right)\varphi}{\bfP(\varphi,\varphi)} 
   \cdot i^{r_{\sLambda,1}-r_{\sLambda,2}} I_\infty(k,\cD_{l_1,l_2},\cD_l,\Lambda_\infty)\\
   &\times E_p\left(k+\frac{l+l_1+l_2}{2},\Pi_x\times\pi_x\times \chi\right)\cdot L^S\left(k+\frac{l+l_1+l_2}{2},\Pi_x\times\pi_x\times \chi\right).
\end{align*}
Here 
\begin{enumerate}[leftmargin=2em]
\item[--]  $\sS_{\GL(2),x}$ (resp. $\sS_{\GSp(4),x})$ is an orthogonal basis  of the space spanned by ordinary cuspidal holomorphic forms on $\GL(2)$ of weight $l$ and tame level $K^p_{\GL(2)}$ (resp. $\GSp(4)$ of weight $(l_1,l_2)$ and tame level $K^p_{\GSp(4)}$) with nebentypus at $p$ given by \eqref{eq:pneben}, belonging to the Hecke eigenspace parameterized by $x$,
\item[--] $I_\infty(k,\cD_{l_1,l_2},\cD_l,\Lambda_\infty)$ is the archimedean zeta integral given in \cite[(4.2.5)]{pFL} with $k+\frac{\epsilon}{2},t_k$ in {\it loc.cit} equal to $k+\frac{l+l_1+l_2}{2},|2k+l+l_1+l_2-1|+3$ here and $\cD_{l_1,l_2},\cD_l$ holomorphic discrete series of $\GSp(4),\GL(2)$ of weights $(l_1,l_2),l$,
\item[--] $f_\bbc$  denotes the Fourier coefficients of $f$ indexed by $\bbc$,
\item[--] $\Lambda$ is the classical Hecke character corresponding to the specialization of $\bLam$ at $(\tau_1,\tau_2)$, and its $\infty$-type is denoted $(r_{\sLambda,1},r_{\sLambda,2})$.
\end{enumerate}
If further assuming that $l=l_1=l_2$, then 
\begin{align*}
    \mu^S_{\sC_1,\sC_2}(\kappa,x)
    =&\,\bbc\sqrt{\det\bS}\,2^{-3l-4}i^l\sum_{f\in \sS_{\GL(2),x}} \frac{f_{\bbc}f}{\bfP(f,f)}\sum_{\varphi\in \sS_{\GSp(4)},x} \frac{\bes^\dagger_{\bS,\Lambda}\left(\varphi\right)\varphi}{\bfP(\varphi,\varphi)} \\
   &\times E_\infty \Big(k+\frac{3l}{2},\Pi_x\times\pi_x\times \chi\Big) E_p\Big(k+\frac{3l}{2},\Pi_x\times\pi_x\times \chi\Big) L^S\Big(k+\frac{3l}{2},\Pi_x\times\pi_x\times \chi\Big).
\end{align*}
\end{thm}

\begin{proof}
We first examine the evaluations of 
\begin{align*}
   (\mathds{1}_{\sC_1}\otimes \mathds{1}_{\sC_2})\cdot \mu_{\cE} \in &\,\Meas\left(\bQ^\times\backslash\bA^\times_{\bQ,f}/U^p,\cM_{\GL(2),\ord}\wh{\otimes}_{\cO} \cM_{\GSp(4),\ord}\right)\\
   &\hspace{12em}\otimes_{\tilde{\Lambda}_{\GL(2)}\hat{\otimes}_\cO\tilde{\Lambda}_{\GSp(4)}} (F_{\sC_1}\hat{\otimes}_\cO F_{\sC_2}). 
\end{align*}


By the construction of $\mu_\cE$, we know that at $(\kappa,x)$ with $(\kappa,\tau,\tau_1,\tau_2)$ classical,
\begin{align*}
   \left((\mathds{1}_{\sC_1}\otimes \mathds{1}_{\sC_2})\cdot \mu_{\cE}\right)(\kappa,x)
   = \sum_{f\in sS_{\GL(2),x}}\sum_{\varphi\in \sS_{\GSp(4),x}} 
   \frac{\left.\bfP\left(E^\Sieg_{\kappa,\tau,\tau_1,\tau_2}\right|_{\GL(2)\times\GSp(4)},f\otimes\varphi\right)}{\bfP(f,f)\bfP(\varphi,\varphi)}
   f\otimes\varphi.
\end{align*}
(Here we use that the definition of $\bfP(-,-)$ plus the tame level and ordinarity of $f,\varphi$ implies that applying $\bfP(-,f\otimes \varphi)$ to $\mu_\cE(\kappa,x)=e^{\GL(2)}_\ord e^{\GSp(4)}_\ord \,{\rm Proj}_{K^p_{\GL(2)},K^p_{\GSp(4)}}\Big(\left.E^\Sieg_{\kappa,\tau,\tau_1,\tau_2}\right|_{\GL(2)\times\GSp(4)}\Big)$ and $\left.E^\Sieg_{\kappa,\tau,\tau_1,\tau_2}\right|_{\GL(2)\times\GSp(4)}$ produces the same value.) The computation in the proof of \cite[Theorem~4.2.1]{pFL} (or more precisely the formula for $I_p(s)$ and $C_{k,\chi,\Pi,\pi}(s)$ in {\it loc.cit} gives 
\begin{align*}
    &\bfP\left(\left.E^\Sieg_{\kappa,\tau,\tau_1,\tau_2},f\otimes\varphi\right|_{\GL(2)\times\GSp(4)}\right)\\
    =&\,C(k,\chi,x)\cdot f_\bbc \bes^\dagger_{\bS,\Lambda}(\varphi)
    \cdot I_\infty(k,\Pi_{x,\infty},\pi_{x,\infty},\Lambda_\infty)\\
   &\times E_p\left(k+\frac{l+l_1+l_2}{2},\Pi_x\times\pi_x\times \chi\right)\cdot L^S\left(k+\frac{l+l_1+l_2}{2},\Pi_x\times\pi_x\times \chi\right)
\end{align*}
with  $\Lambda$ the classical Hecke character associated to the specialization of $\bLam$ at $(\tau_1,\tau_2)$, and
\begin{equation}\label{eq:factor-C}
\begin{aligned}
   C(k,\chi,x)=&\,\vol_S \, \frac{1-\left(\frac{\cK}{p}\right)p^{-1}}{1+p^{-1}}\,\left|\frac{\alphaS-\alphabS}{\sqrt{\disc(\cK/\bQ)}}\right|_p\cdot |\bbc|^2_p |(\alphaS-\alphabS)^2|^2_p\cdot \chi_p(-1)(-1)^k\\
    &\times(\chi^{-1}_p|\cdot|^{-k}_p)(\bbc) \bbc^{-k} \cdot \xi^{-1}\xi^{-1}_1\xi^{-1}_2(\bbc|\bbc|_p) (\bbc|\bbc|_p)^{-l-l_1-l_2}\cdot \lambda_{\GL(2)}\lambda_{\GSp(4)}(x)^{-\val_p(\bbc)} \\
    &\times(\chi^{-1}_p|\cdot|^{-k}_p)(-(\alphaS-\alphabS)^2)(-(\alphaS-\alphabS)^2)^{-k} \\
    &\times \xi^{-1}(-(\alphaS-\alphabS)^2|\alphaS-\alphabS|^2_p)(-(\alphaS-\alphabS)^2|\alphaS-\alphabS|^2_p)^{-l}\\
    &\times\Lambda^{-c}_p(-(\alphaS-\alphabS))(\alphaS-\alphabS)^{-r_{\sLambda_1}}(-\alphaS+\alphabS)^{-r_{\sLambda_2}}\cdot\lambda_{\GL(2)}(x)^{-\val_p((\alphaS-\alphabS)^2)}\cdot i^{r_{\sLambda,1}-r_{\sLambda,2}},
\end{aligned}
\end{equation}
where $\vol_S$ is a nonzero constant independent of $k,\chi,x$ (and can be expressed in terms of the volumes of some compact subgroups at $v\in S$), $\lambda_{\GL(2)}\in \bI^\times_{\sC_1},\lambda_{\GSp(4)}\in\bI^\times_{\sC_2}$ denote the eigenvalues of the $\bU_p$-operators corresponding to $\begin{bmatrix}p\\&1\end{bmatrix},\begin{bmatrix}p\cdot\bid_2\\&\bid_2\end{bmatrix}$ along $\sC_1,\sC_2$.  (To plug in the formulas in \cite{pFL}, $s,k+\frac{\epsilon}{2},\Lambda,\eta_{\pi_p,1}(a),\eta_{\sPi_p,3}(a)$ for $a\in\bQ_p$  in {\it loc.cit} correspond to $k+\frac{l+l_1+l_2-1}{2}, k+\frac{l+l_1+l_2}{2},{\Lambda^{-c}|\cdot|^{\frac{l_1+l_2}{2}}_{\bA_\cK}}$, $\xi^{-1}(a|a|_p)\cdot(p^{\frac{l-1}{2}}\lambda^{-1}_{\GL(2)}(x))^{\val_p(a)}$, $\xi^{-1}_1\xi^{-1}_2(a|a|_p)\cdot (p^{\frac{l_1+l_2-3}{2}}\lambda^{-1}_{\GSp(4)}(x))^{\val_p(a)}$ here. Also,  note that the integral in {\it loc.cit} is over $\GU(1,1)\times\GSp(4)$ with $f$ extended to $\GU(1,1)$ by $\Upsilon$ equals our integral over $\GL(2)\times\GSp(4)$ here because the central characters match.) 

From the above formula for $C(k,\chi,x)$, it is easy to see that there exists 
\[
   \cC\in \left(\cO\llb \bQ^\times\backslash\bA^\times_{\bQ,f}/U^p\rrb \hat{\otimes}_{\cO} \bI_{\sC_1}\hat{\otimes}_\cO \bI_{\sC_2}\otimes_\cO F\right)^\times
\]
such that for all $(\kappa,x)$ with $(\kappa,\tau,\tau_1,\tau_2)$ classical,
\[
   \cC(\kappa,x)=i^{-r_{\sLambda,1}+r_{\sLambda,2}}\cdot C(k,\chi,x)
\]
Let
\begin{align*}
    \mu^S_{\sC_1,\sC_2}= \cC^{-1}\left((\mathds{1}_{\sC_1}\otimes \mathds{1}_{\sC_2})\cdot \mu_{\cE}\right)
  \in &\,\Meas\left(\bQ^\times\backslash\bA^\times_{\bQ,f}/U^p,\cM_{\GU(1,1),\ord}\wh{\otimes}_{\cO} \cM_{\GSp(4),\ord}\right)\\
  &\hspace{2em}\otimes\otimes_{\tilde{\Lambda}_{\GL(2)}\wh{\otimes}_\cO\tilde{\Lambda}_{\GSp(4)}} (F_{\sC_1}\hat{\otimes}_\cO F_{\sC_2}).
\end{align*}
Then for $(\kappa,x)$ as in the statement of the theorem,
\begin{align*}
   \mu^S_{\sC_1,\sC_2}(\kappa,x)
   =&\,\sum_{f\in \sS_{\GL(2),x}}\sum_{\varphi\in \sS_{\GSp(4),x}}
    \frac{f_\bbc \bes^\dagger_{\bS,\Lambda}(\varphi)}{\bfP(f,f)\bfP(\varphi,\varphi)}
    \cdot i^{r_{\sLambda,1}-r_{\sLambda,2}}I_\infty(k,\cD_{l_1,l_2},\cD_l,\Lambda_\infty)\\
   &\times E_p\left(k+\frac{l+l_1+l_2}{2},\Pi_x\times\pi_x\times \chi\right)\cdot L^S\left(k+\frac{l+l_1+l_2}{2},\Pi_x\times\pi_x\times \chi\right)
   \cdot f\otimes\varphi   
\end{align*}


\end{proof}

\section{Specialization to Hida families of Yoshida lifts}
In order to get a complete interpolation formula for the four-variable $p$-adic $L$-function $\mu^S_{\sC_1,\sC_2}$ constructed in Theorem~\ref{thm:pFL1}, we calculate the archimedean zeta integral $I_\infty(k,\cD_{l_1,l_2},\cD_l,\Lambda_\infty)$ by putting $\sC_2=\theta(\sB,\sB')$, the Hecke eigensystem associated to the Yoshida lifts of Hida families $\sB,\sB'$ on $\GL(2)$ and comparing  $\mu^S_{\sC_1,\theta(\sB,\sB')}$ with some previously constructed $p$-adic $L$-functions.

\subsection{Some previous results on $p$-adic $L$-functions}\label{sec:pL}

For simplicity, in this section and \S\ref{sec:compare}, we assume that there exist finite places $v\neq p$ such that $K^p_{\GL(2),v}\neq \GL_2(\bZ_v)$ and for all such $v$'s,  
\[
    K^p_{\GL(2),v}=\left\{g\in\GL(2,\bZ_v):g\equiv \begin{bmatrix}\ast&\ast\\0&\ast\end{bmatrix}\mod \varpi_v\right\}.
\]

We recall some previous results on constructions of Kubota--Leopold $p$-adic $L$-functions, Rankin--Selberg $p$-adic $L$-functions, and $p$-adic (degree 5) standard $L$-functions for $\Sp(4)$. To simplify the writing of the interpolation properties, we use the following convention: Given an automorphic representation $\sigma$ with $\sigma_\infty$ isomorphic to holomorphic discrete series and ordinary at $p$, we let
\[
    D^S(s,\sigma)=E_\infty(s,\sigma)\,E_p(s,\sigma)\,L^S(s,\sigma)
\]
with $E_\infty(s,\sigma),E_p(s,\sigma)$ the modified Euler factor at $\infty$ and $p$ for $p$-adic interpolation as defined in \cite{CoPerrin,Coates}.

\subsubsection{Kubota--Leopoldt $p$-adic $L$-function}
\begin{thm}\label{thm:KLp}
There exists
\[
    \cL^S_{\KL}\in \Meas\left(\bQ^\times\backslash\bA^\times_\bA/U^p,\cO\right)
\]
such that for all arithmetic $\kappa=(k,\chi)\in \Hom_\cont\left(\bQ^\times\backslash\bA^\times_{\bQ,f}/U^p,\bar{\bQ}_p\right)$ with $\chi(-1)=(-1)^k$ and $k\geq 1$ or $\chi(-1)=(-1)^{k+1}$ and $k\leq 0$ ,
\begin{align*}
  \cL^S_\KL(\kappa)=D^S(k,\chi).
\end{align*}
\end{thm}
(Here,  $S$ is assumed to  contain finite places other than $p$. Hence, the imprimitive Kubota--Leopoldt $p$-adic $L$-function does not have poles.)

\subsubsection{Rankin--Selberg $p$-adic $L$-function}

\begin{thm}\label{thm:Rankin}
Let $\sB_1,\sB_2\subset \Spec(\bT_{\GL(2),\ord})$ be two geometrically irreducible components. We assume that $\sB_1$ is {\it primitive}, {\it i.e.} the newforms in the automoephic representations corresponding to classical points of $\sB_1$ has tame level equal to $K^p_{\GL(2)}$). Denoting by $F_{\sB_1},F_{\sB_2}$ the function fields of $\sB_1,\sB_2$, there exists
\[
   \cL^S_{\sB_1,\sB_2}\in \Meas\left(\bQ^\times\backslash\bA^\times_\bA/U^p,\Lambda_{\GL(2)}\hat{\otimes}_\cO\Lambda_{\GL(2)}\right)\otimes_{\Lambda_{\GL(2)}\hat{\otimes}_\cO\Lambda_{\GL(2)}} (F_{\sB_1}\hat{\otimes}_\cO F_{\sB_2})
\]
satisfying the interpolation property: Suppose that $(x_1,x_2)\in \sB_1(\bar{\bQ}_p)\times\sB_2(\bar{\bQ}_p)$ is a classical  point of weights $t_1> t_2\geq 2$  where the weight projection map  $\Lambda_{\GL(2)}\hat{\otimes}_\cO \Lambda_{\GL(2)}\ra \bT_{\GL(2),\ord}\hat{\otimes}_\cO\bT_{\GL(2),\ord}$ is \'{e}tale. Then $ \cL^S_{\sB_1,\sB_2}$ has no poles at $(x_1,x_2)$, and for  an arithmetic character $\kappa=(k,\chi)\in \Hom_\cont\left(\bQ^\times\backslash\bA^\times_{\bQ,f}/U^p,\bar{\bQ}^\times_p\right)$ such that 
\[
   -t_1+1\leq k\leq -t_2
\]
{\it i.e.} $s=k+\frac{t_1+t_2}{2}$ is a critical point for the Rankin--Selberg $L$-function $L(s,\sigma_{x_1}\times\sigma_{x_2})$, we have
\begin{align*}
     \cL^S_{\sB_1,\sB_2}(\kappa,x)=\frac{D^S\left(k+\frac{t_1+t_2}{2},\sigma_{x_1}\times\sigma_{x_2}\times\chi\right)}{ (-2i)^{t_1+1}\bfP(f_{x_1},f_{x_1})}
\end{align*}
where $\sigma_{x_j}$, $j=1,2$, is the (unique) unitary automorphic representation of $\GL(2)$ (with unitary central character) giving rise to the Hecke eigensystem parameterized by $x_j$ (up to a twist by a real power of $|\det|$), and $f_{x_1}\in\sigma_{x_1}$ is the normalized eigenform for the Hecke eigensystem parameterized by $x_1$. (The modified Petersson inner product $\bfP(f_{x_1},f_{x_1})$ is defined as in \eqref{eq:Pet2}.)
\end{thm}

This theorem is proved in \cite{Hi88} (Theorem~5.1d.) ({\it cf.} also \cite[Theorem~A]{Chen-Hsieh}. Our $\cL^S_{\sB_1,\sB_2}$ are obtained from the $p$-adic $L$-functions in {\it loc.cit} by removing $L$-factors at $v\in S-\{p,\infty\}$ which are $p$-adically interpolatable and by a change of variable. To see that $\sL^S_{\sB_1,\sB_2}$ can be obtained by a change of variable from the $p$-adic $L$-functions in {\it loc.cit}, also note that there is a slight difference between our convention of nebentypus here and that in {\it loc.cit}. If $\sigma_j$ has central character $\omega_j$, then in our convention the $p$-nebentypus is $\omega_j|_{\bZ^\times_p}$ and in the conventions in {\it loc.cit}, the nebentypus sends $q_v$ to $\omega_v(q_v)$, so essentially is $\omega^{-1}_j|_{\bZ^\times_p}$.)

\subsubsection{Standard $p$-adic $L$-function for Yoshida lifts}

Let $\sB,\sB'\subset \Spec(\bT_{\GL(2),\ord})$ be two geometrically irreducible components. Let $F_{\sB},F_{\sB'}$ be the function fields of $\sB,\sB'$ and $\bI_{\sB},\bI_{\sB'}$  be the integral closures of $\Lambda_{\GL(2)}$ in $F_{\sB},F_{\sB'}$. Denote by
\begin{align*}
   &\lambda_{\sB}:\bT_{\GL(2),\ord}\lra \bI_{\sB} ,
   &&\lambda_{\sB'}:\bT_{\GL(2),\ord}\lra \bI_{\sB'}
\end{align*}
the corresponding Hecke eigensystems, and 
\begin{align*}
   \omega_{\sB},\omega_{\sB'}:\bQ^\times\backslash\bA^\times_{\bQ,f}\lra \Lambda^\times_{\GL(2)}
\end{align*} 
the central characters. We fix a square root of them
\[
    \omega^{\frac{1}{2}}_{\sB},\omega^{\frac{1}{2}}_{\sB'}:\bQ^\times\backslash\bA^\times_{\bQ,f}\lra \Lambda^\times_{\GL(2)}.
\]

We have the group homomorphism
\begin{align*}
    T^1_{\GL(2)}\times T^1_{\GL(2)}&\lra T^1_{\GSp(4)},\\
   \big(\diag(a_1,a^{-1}_1),\diag(a_2,a^{-1}_2)\big)&\longmapsto \diag(a_1a_2,a_1a^{-1}_2,a^{-1}_1a^{-1}_2,a^{-1}_1a_2)
\end{align*}
which induces 
\[
   \wt{\Lambda}_{\GL(2)}\hat{\otimes}\wt{\Lambda}_{\GL(2)} \lra  \wt{\lambda}_{\GSp(4)}.
\] 
Let
\[
    \bI_{\theta(\sB,\sB')}= \wt{\Lambda}_{\GSp(4)}\otimes_{\wt{\Lambda}_{\GL(2)}\hat{\otimes}\wt{\Lambda}_{\GL(2)}} (\bI_{\sB}\hat{\otimes}_{\cO}\bI_{\sB'}).
\]
It follows from the theory of theta lifts \cite{RobGlobalTheta} that if there exists a finite place $v\neq p$ such that the classical specializations of $\sB,\sB'$ are discrete series at $v$, then 
for suitable $K^p_{\GSp(4)}$, there exists a geometrically irreducible component $\theta(\sB,\sB')\subset \Spec(\bT_{\GSp(4),\ord})$ with the $\wt{\Lambda}_{\GSp(4)}$-algebra homomorphism
\begin{align*}
   \lambda_{\theta(\sB,\sB')}:\bT_{\GSp(4),\ord}\lra  \bI_{\theta(\sB,\sB')}
\end{align*}
such that the central character equals $\bQ^\times\backslash\bA^\times_f\stackrel{\omega_{\sB}}{\lra}\Lambda^\times_{\GL(2)}\lra \Lambda^\times_{\GSp(4)}$ where the second map is induced by $T^1_{\GL(2)}\ra T^1_{\GSp(4)}, \,\diag(a,a^{-1})\mapsto \diag(a,a,a^{-1},a^{-1})$, and
{\small
\begin{align*}
   \lambda_{\theta(\sB,\sB')}\left(\GSp(4,\bZ_v)\begin{bmatrix}\varpi_v\\ &\varpi_v\\&&1\\&&&1\end{bmatrix}\GSp(4,\bZ_v)\right)&=\lambda_{\sB}\left(\GL(2,\bZ_v)\begin{bmatrix}\varpi_v\\&1\end{bmatrix}\GL(2,\bZ_v)\right),\\
   \lambda_{\theta(\sB,\sB')}\left(\GSp(4,\bZ_v)\begin{bmatrix}\varpi_v\\ &1\\&&1\\&&&\varpi_v\end{bmatrix}\GSp(4,\bZ_v)\right)&=\omega^{\frac{1}{2}}_{\sB}\,\omega^{-\frac{1}{2}}_{\sB'}(\varpi_v)\lambda_{\sB'}\left(\GL(2,\bZ_v)\begin{bmatrix}\varpi_v\\&1\end{bmatrix}\GL(2,\bZ_v)\right).
\end{align*}}
(It follows from the results in \cite{RobGlobalTheta}  that for all classical specializations $\sigma\not\cong\sigma'$ of $\sB,\sB'$ of weights $t>t'\geq 2$ such that they are both discrete series at a finite place $v$ and the product of the central characters is a square, there is a nonzero Yoshida lift of $\sigma\boxtimes\sigma'\otimes (\omega_{\sigma},\omega_{\sigma'})^{\frac{1}{2}}\circ\det$ to $\GSp(4)$ with archimedean component isomorphic to a holomorphic discrete series. This implies the existence of the geometrically irreducible component $\theta(\sB,\sB')\subset \Spec(\bT_{\GSp(4),\ord})$.) Let $F_{\theta(\sB,\sB')}={\rm Frac}\left(\bI_{\theta(\sB,\sB')}\right)$.

\begin{thm}\label{thm:pstd}
There exists 
\[
   \mu^S_{\theta(\sB,\sB')}\in \Meas\left(\bQ^\times\backslash\bA^\times_{\bQ,f}/U^p,\cM_{\GSp(4),\ord}\otimes_{\wt{\Lambda}_{\GSp(4)}}\cM_{\GSp(4),\ord}\right)\otimes_{\wt{\Lambda}_{\GSp(4)}} F_{\theta(\sB,\sB')}
\]
satisfying the interpolation property: Suppose that $x\in \theta(\sB,\sB')(\bar{\bQ}_p)$ is point where the weight projection map $\wt{\Lambda}_{\GSp(4)}\ra \bT_{\GSp(4),\ord}$ is \'{e}tale and has arithmetic image $(l_1,l_2,\xi_1,\xi_2)$ with $l_1\geq l_2\geq 3$. If $\kappa=(k,\xi)\in\Hom_\cont\left(\bQ^\times\backslash\bA^\times_{\bQ,f}/U^p,\bar{\bQ}^\times_p\right)$ is an arithmetic point with $t_2\geq k\geq 3$ and $\chi(-1)=(-1)^k$, then
\begin{align*}
   \mu^S_{\theta(\sB,\sB')}(\kappa,x)
   =&\,2^{-l_1-l_2} \cdot D^S(k-2,\theta(\sB,\sB')_x\times\chi)
 \sum_{\varphi\in \sS_{\GSp(4),x}}\frac{\varphi\boxtimes\varphi}{\bfP(\varphi,\varphi)} \\
  =&\, 2^{-l_1-l_2} \cdot D^S(k-2,\chi)\,D^S\left(k-2,\sigma_x\times\sigma'_x\times \omega^{-\frac{1}{2}}_{\sigma_x}\omega^{-\frac{1}{2}}_{\sigma'_x}\chi\right)
 \sum_{\varphi\in \sS_{\GSp(4),x}}\frac{\varphi\boxtimes\varphi}{\bfP(\varphi,\varphi)} ,
\end{align*}
where $\sigma_x,\sigma'_x$ are the (unique) automorphic representations of $\GL(2)$ giving rise to the Hecke eigensystems parameterized by the point in $\sB(\bar{\bQ}_p)\times\sB'(\bar{\bQ}_p)$  induced by $x$ and the natural map $\bI_{\sB}\hat{\otimes}_\cO\bI_{\sB'}\ra \bI_{\theta(\sB,\sB')}$, the set $\sS_{\GSp(4),x}$ is an orthogonal basis of the space spanned by ordianry cuspidal holomorphic Seigel modular forms on $\GSp(4)$ of weight $(l_1,l_2)$, tame level $K^p_{\GSp(4)}$ belonging to the Hecke eigenspace parameterized by $x$, or equivalently of $\theta(\sigma_x,\sigma'_x)$. 
\end{thm}

\begin{proof}
We apply the construction in \cite{SLF} to the special case of $\theta(\sB,\sB')$ on $\GSp(4)$. The construction of the measure $\mu^S_{\theta(\sB,\sB')}$ is a special case of the construction described in the paragraph containing the interpolation formula~(7.0.1) on page 58. (The archimedean zeta integral is left as an uncomputed factor in {\it loc.cit}. It has been calculated in \cite{AZI} and verified to agree with what is expected according to the conjecture of Coates and Perrin--Riou on $p$-adic $L$-functions.) The last factor in the formula for $\mu^S_{\theta(\sB,\sB')}$ here is  slightly different from the formula~(7.0.1) in \cite{SLF} because the interpolation formula is computed by applying $\left<\,,\bar{\varphi}\right>$ to the specialization. If we apply instead $\bfP(\,,\varphi)$ to the specialization, we get the above formula.  
\end{proof}

\subsection{Comparison of $p$-adic $L$-functions}\label{sec:compare}
Let $\sB,\sB'$ be primitive geometrically irreducible components of $\Spec(\bT_{\GL(2),\ord})$ such that $\Spec(\bT_{\GSp(4),\ord})$ has a geometrically irreducible component $\theta(\sB,\sB')$ . By using the results on $p$-adic $L$-functions in \S\ref{sec:pL},  we can deduce the following proposition.

\begin{prop}\label{prop:modpLf}
There exists 
\begin{align*}
    \cL^{S,\ast}_{\sB,\sC_1},\cL^{S,\ast}_{\sC_1,\sB'}
    \in \Meas\left(\bQ^\times\backslash\bA^\times_{\bQ,f}/U^p,\Lambda_{\GL(2),\ord}\wh{\otimes}_{\cO} \Lambda_{\GSp(4)}\right)\otimes_{\wt{\Lambda}_{\GL(2)}\hat{\otimes}_\cO\wt{\Lambda}_{\GSp(4)}} (F_{\sC_1}\hat{\otimes}_\cO F_{\theta(\sB,\sB')})
\end{align*}
and 
\[
   \cF_{\theta(\sB,\sB)} \in \left(\cM_{\GSp(4),\ord}\otimes_{\wt{\Lambda}_{\GSp(4)}} \cM_{\GSp(4),\ord}\right)\otimes_{\wt{\Lambda}_{\GSp(4)}}  F_{\theta(\sB,\sB')})
\]
satisfying the following interpolation properties: In the setting of Theorem~\ref{thm:pFL1} with $\sC_2=\theta(\sB,\sB')$ and writing $x=(x_1,x_2)\in \sC_1(\bar{\bQ}_p)\times \theta(\sB,\sB')(\bar{\bQ}_p)$,
\begin{align*}
   \cL^{S,\ast}_{\sB,\sC_1}(\kappa,x)&=\frac{D^S\left(k+\frac{l+l_1+l_2}{2},\sigma_x\times\pi_x\times\chi\right)}{(-2i)^{l_1+l_2-1}\bfP(h_x,h_x)}\\
   \cL^{S,\ast}_{\sC_1,\sB'}(\kappa,x)&=\frac{D^S\left(k+\frac{l+l_1+l_2}{2},\pi_x\times\sigma'_x\times\omega^{\frac{1}{2}}_{\sigma_x}\omega^{-\frac{1}{2}}_{\sigma'_x}\chi\right)}{(-2i)^{l+1}\bfP(f_x,f_x)}
\end{align*}
and 
\begin{align*}
   \cF_{\theta(\sB,\sB')}(x_2)= 2^{-1}i^{-l_1-l_2+1}\bfP(h_x,h_x)\sum_{\varphi\in\sS_{\GSp(4),x}} \frac{\varphi\boxtimes\varphi}{\bfP(\varphi,\varphi)}
\end{align*}
with $\pi_x,f_x\in\pi_x$ as in Theorem~\ref{thm:pFL1}, $\sigma_x,\sigma'_x$ as in Theorem~\ref{thm:pstd}, and $h_x\in\sigma_x$  the unique normalized ordinary form fixed by $K^p_{\GL(2)}$. (Note that the weights of the archimedean components of $\sigma_x,\sigma'_x$ are $l_1+l_2-2,l_1-l_2+2$.)
\end{prop}

\begin{proof}
Applying pullback and change of variable to the Rankin--Selberg $p$-adic $L$-function in Theorem~\ref{thm:Rankin} shows the existence of $\cL^{S,*}_{\sB,\sC},\cL^{S,\ast}_{\sC_1,\sB'}$, as well as the existence of 
\[
\cL^{S,\ast}_{\sB,\sB'}\in \Meas\left(\bQ^\times\backslash \bA^\times_{\bQ,f}/U^p,\bI_{\theta(\sB,\sB'}\right)\otimes_{\bI_{\theta(\sB,\sB')}} F_{\theta(\sB,\sB')}
\]
such that for $(\kappa,x_2)\in \Hom_\cont(\bQ^\times\backslash \bA^\times_{\bQ,f}/U^p,\bar{\bQ}^\times_p)\times \theta(\sB,\sB')(\bar{\bQ}_p)$ satisfying the conditions in Theorem~\ref{thm:pstd},
\begin{align*}
   \cL^{S,*}_{\sB,\sB'}(\kappa,x)=\frac{D^S\left(k-2,\sigma_{x}\times\sigma'_{x}\times \omega^{-\frac{1}{2}}_{\sigma_{x}}\omega^{-\frac{1}{2}}_{\sigma'_{x}}\chi\right)}{(-2i)^{l_1+l_2-1}\bfP(h_{x},h_{x})}.
\end{align*}
A change of variable to the Kubota--Leopoldt $p$-adic $L$-function recalled in Theorem~\ref{thm:KLp} gives $\cL^{S,\ast}_{\rm KL}\in  \Meas\left(\bQ^\times\backslash\bA^\times_\bA/U^p,\cO\right)$ such that for $\kappa$ as as in Theorem~\ref{thm:pstd}
\[
   \cL^{S,\ast}_{\rm KL}(\kappa)=D^S(k-2,\chi).
\]
The desired $\cF_{\theta(\sB,\sB')}$ can be obtained as $(\cL^{S,\ast}_{\sB,\sB'}\cL^{S,\ast}_{\rm KL})^{-1} \mu^S_{\theta(\sB,\sB')}$ with $\mu^S_{\theta(\sB,\sB')}$ as in Theorem~\ref{thm:pstd}.
\end{proof}

Next, we show that we can take the $(\bS,\bLam)$-Bessel coefficient on the second factor of $\cF_{\theta(\sB,\sB')}$.

\begin{prop} \label{prop:intBes}
Take nonzero $\cH\in \bI_{\theta(\sB,\sB')}$ such that $\cH \cF_{\theta(\sB,\sB')}\in \cM_{\GSp(4),\ord}\otimes_{\wt{\Lambda}_{\GSp(4)}} \bI_{\theta(\sB,\sB')}$. Given $\bS=\begin{bmatrix}\bba&\frac{\bbb}{2}\\\frac{\bbb}{2}&\bbc\end{bmatrix}\in\Sym_2(\bQ)_{>0}$ and $\bLam\in\Hom_\cont(\cK^\times\backslash\bA^\times_{\cK,f},\Lambda^\times_{\GSp(4)})$ extending $\omega_{\theta(\sB,\sB')}=\omega_{\sB}$, where $\cK=\bQ(\sqrt{-\det\bS})$, there exists 
\[
   \cF_{\theta(\sB,\sB'),\bS,\bLam}
   \in \cM_{\GSp(4),\ord}\otimes_{\wt{\Lambda}_{\GSp(4)}} F_{\theta(\sB,\sB')}
\]
such that for $x\in\theta(\sB,\sB')(\bar{\bQ}_p)$ as in Proposition~\ref{prop:modpLf} and not a pole of $\cH$,
\[
   \cF_{\theta(\sB,\sB'),\bS,\bLam}(x)=2^{-1}i^{-l_1-l_2+1}\bfP(h_x,h_x)\sum_{\varphi\in\sS_{\GSp(4),x}} \frac{\bes^\dagger_{\bS,\Lambda}(\varphi)\varphi}{\bfP(\varphi,\varphi)}.
\]

\end{prop}
\begin{proof}
We follow the method in \cite[\S10.2]{HsiehYama} to construct the desired $\cF_{\theta(\sB,\sB'),\bS,\bLam}$ from the $\cF_{\theta(\sB,\sB')}$ in Proposition~\ref{prop:modpLf}. Fix $c\geq 0$ such that $p^c\alphaS\in \cO_{\cK}$ and an open compact subgroup $U^p_\cK\subset\bA^{\times,p}_\cK$ such that $\bLam$ factors through the quotient by $U^p_\cK$. Given a positive integer $n$, we let 
\begin{align*}
    U_{\cK,p,n}&=\bZ^\times_p (1+p^{n+c}\bZ_p\alphaS),
    &U_{\cK,n}&=U^p_\cK U_{\cK,p,n},
\end{align*}
and
\[
   \rho_n:\wt{\Lambda}_{\GSp(4)}\lra  \cO\big[\,T^1_{\GSp(4)}(\bZ/p^n\bZ_)\,\big]
\]
be the natural projection induced by $T^1_{\GSp(4)}(\bZ_p)\ra T^1_{\GSp(4)}(\bZ/p^n\bZ)$. Put
\begin{align*}
   \bI_{\theta(\sB,\sB'),n}&=\bI_{\theta(\sB,\sB')}\otimes_{\wt{\Lambda}_{\GSp(4)},\rho_n} \cO\big[\,T^1_{\GSp(4)}(\bZ/p^n\bZ)\,\big].
\end{align*}
Then $\rho_n$ naturally induces $\rho_n:\bI_{\theta(\sB,\sB')}\ra \bI_{\theta(\sB,\sB'),n}$. Taking the $q$-expansion of the second factor at $\begin{bmatrix}A\\&D\end{bmatrix}\in \GSp(4,\bA_{\bQ,f})$ and taking the coefficient indexed by $\bS$ gives a map
\[
   \varepsilon_{\qexp,\bS}\left(\,\cdot\,, \begin{bmatrix}A\\&D\end{bmatrix}\right):\cM_{\GSp(4),\ord}\otimes_{\wt{\Lambda}_{\GSp(4)}}\cM_{\GSp(4),\ord}\lra \cM_{\GSp(4)}.
\]
Define $\Theta_n\in \cM_{\GSp(4),\ord}\otimes_{\wt{\Lambda}_{\GSp(4)}} \bI_{\theta(\sB,\sB'),n}$ as
\begin{align*}
   \Theta_n=
   \hspace{-1em}\sum_{\fz\in \cK^\times\bA^\times_{\bQ,f}\backslash\bA^\times_{\cK,f}/U_{\cK,n}} \hspace{-1em}\rho_n\left(\lambda^{-n-c}\cdot\bLam(\fz)^{-1} \,\varepsilon_{\qexp,\bS}\left(\cH \cF_{\theta(\sB,\sB')},\begin{bsm}\imath_\bS(\fz)\\&\ltrans{\imath_\bS(\bar{\fz})}\end{bsm}\begin{bsm}p^{2}\\&p\\&&1\\&&&p\end{bsm}^{n+c}_p\right)\right)[\fz],
\end{align*}
with $\lambda$ equal to the product of $\omega_{\sB,p}(p)$ and the eigenvalue of the $\bU_p$-operator associated to $\begin{bsm}p\\&1\\&&p^{-1}\\&&&1\end{bsm}$ corresponding to $\theta(\sB,\sB')$.
For $y\in\bZ_p$, we have
\[
    \imath_\bS(1+p^{n+c}y\alphaS)
    =\begin{bsm}1&p^{n+c}y\\&1\\ &&1\\&&-p^{n+c}y&1\end{bsm}\begin{bsm}(1+p^{n+c}y\alphaS)(1+p^{n+c}y\alphabS)\\-p^{n+c}y\alphaS\alphabS&1\\ &&1&p^{n+c}y\alphaS\alphabS\\ && &(1+p^{n+c}y\alphaS)(1+p^{n+c}y\alphabS)\end{bsm}
\]
and
\begin{align*}
   &\rho_n\left(\varepsilon_{\qexp,\bS}\left(\sum_{y\in\bZ/p} \cH\cF_{\theta(\sB,\sB')},\begin{bsm}\imath_\bS(\fz(1+p^{n+c}\alphaS y))\\&\ltrans{\imath_\bS}(\bar{\fz}(1+p^{n+c}\alphabS y))\end{bsm}\begin{bsm}p^{2}\\&p\\&&1\\&&&p\end{bsm}^{n+1+c}_p\right)\right)\\
   =&\rho_n\left(\varepsilon_{\qexp,\bS}\left(\sum_{y\in\bZ/p} \cH\cF_{\theta(\sB,\sB')},\begin{bsm}\imath_\bS(\fz)\\&\ltrans{\imath_\bS(\bar{\fz})}\end{bsm}\begin{bsm}1&p^{n+c}y\\&1\\ &&1\\&&-p^{n+c}y&1\end{bsm}\begin{bsm}p^{2}\\&p\\&&1\\&&&p\end{bsm}^{n+1+c}_p\right)\right)\\
   =&\rho_n\left(\varepsilon_{\qexp,\bS}\left(\sum_{y\in\bZ/p} \cH\cF_{\theta(\sB,\sB')},\begin{bsm}\imath_\bS(\fz)\\&\ltrans{\imath_\bS(\bar{\fz})}\end{bsm}\begin{bsm}p^{2}\\&p\\&&1\\&&&p\end{bsm}^{n+c}_p\begin{bsm}1&y\\&1\\ &&1\\&&-y&1\end{bsm}\begin{bsm}p^{2}\\&p\\&&1\\&&&p\end{bsm}_p\right)\right)\\
   =&\rho_n\left(\lambda_{\GSp(4),2,1}\cdot \varepsilon_{\qexp,\bS}\left( \cH\cF_{\theta(\sB,\sB')},\begin{bsm}\imath_\bS(\fz)\\&\ltrans{\imath_\bS(\bar{\fz})}\end{bsm}\begin{bsm}p^{2}\\&p\\&&1\\&&&p\end{bsm}^{n+c}_p\right)\right),
\end{align*}
from which it follows that
\[
    \rho_n(\Theta_{n+1})=\Theta_n.
\]
Therefore, the $\Theta_n$'s define an element $\Theta\in \cM_{\GSp(4),\ord}\otimes_{\wt{\Lambda}_{\GSp(4)}}\bI_{\theta(\sB,\sB')}$. Notice that for ordinary $\varphi$ invariant under $\{g\in\GSp(4,\bZ_p):g\mod p^n \in U_{\GSp(4)}(\bZ/p^n)\}$, $\bes^\dagger_{\bS,\Lambda}(\varphi)$ can be computed with $m_1=n+c,m_2=0$ in \eqref{eq:bes-dagger}. Then from the definition of $\Theta_n$'s, we see that up to a scalar, $\cH^{-1}\Theta$ gives the desired $\cF_{\theta(\sB,\sB'),\bS,\bLam}$.

\end{proof}

\begin{prop}\label{prop:compare}
Suppose that $\sC_1$ is primitive and $\bm{f}\in \cM_{\GL(2),\ord}$ is the Hida family corresponding to $\sC_1$ normalized such that the first Fourier coefficient is $1$. Then
\begin{equation}\label{eq:compare}
   \mu^S_{\sC_1,\theta(\sB,\sB')} 
   =\bbc\sqrt{\det\bS}\,2^{-3}i^{-1} {\bm f}_{\bbc}{\bm f}\cdot \cL^{S,\ast}_{\sB,\sC_1}\cL^{S,\ast}_{\sC_1,\sB'} \cdot \cF_{\theta(\sB,\sB'),\bS,\bLam}
\end{equation}
\end{prop}
\begin{proof}
It suffices to check that in the setting of Theorem~\ref{thm:pFL1} with $\sC_2=\theta(\sB,\sB')$, the evaluations of both sides agree at all $(\kappa,x)$ with $l_1=l_2=l$.  It follows from Propositions~\ref{prop:modpLf},\ref{prop:intBes} that
\begin{align*}
    \text{RHS}(\kappa,x)
   =&\,\bbc\sqrt{\det\bS}\,2^{-3}i^{-1} f_{x,\bbc}f_x\cdot 2^{-3l-1}i^{l+1}\\
   &\times\frac{D^S\left(k+\frac{3l}{2},\sigma_x\times\pi_x\times\chi\right) D^S\Big(k+\frac{3l}{2},\pi_x\times\sigma'_x\times\omega^{\frac{1}{2}}_{\sigma_x}\omega^{-\frac{1}{2}}_{\sigma'_x}\times\chi\Big)}{\bfP(f_x,f_x)}
   \hspace{-1em} \sum_{\varphi\in \sS_{\GSp(4),x}}\frac{\bes^\dagger_{\bS,\Lambda}(\varphi)\varphi}{\bfP(\varphi,\varphi)},
\end{align*}
where $f_x$ denotes the specialization of ${\bm f}$ at $x$. Note that when $\Pi_x=\theta(\sB,\sB')_x$, the $L$-function for $\Pi_x$ decomposes as the product of $L$-functions for $\sigma_x$ and $\sigma'_x\otimes  \omega^{\frac{1}{2}}_{\sigma_x}\omega^{-\frac{1}{2}}_{\sigma'_x}\circ\det$, and we have
\[
    D^S(s,\Pi_x\times\pi_x\times\chi)=D^S\left(s,\sigma_x\times\pi_x\times\chi\right) D^S\Big(s,\pi_x\times\sigma'_x\times\omega^{\frac{1}{2}}_{\sigma_x}\omega^{-\frac{1}{2}}_{\sigma'_x}\times\chi\Big).
\]
Hence,
\begin{align*}
   \text{RHS}(\kappa,x)=&\,\bbc\sqrt{\det\bS}\, 2^{-3l-4} i^{l} \cdot D^S\left(k+\frac{3l}{2},\Pi_x\times\pi_x\times\chi\right)\cdot  \frac{f_{x,\bbc}f_x}{\bfP(f_x,f_x)}\sum_{\varphi\in \sS_{\GSp(4),x}}\frac{\bes^\dagger_{\bS,\Lambda}(\varphi)\varphi}{\bfP(\varphi,\varphi)}
\end{align*}
which equals exactly $\mu^S_{\sC_1,\theta(\sB,\sB')}$ by the formula in Theorem~\ref{thm:pFL1}.
\end{proof}

\subsection{The four-variable $p$-adic $L$-function and its interpolation formula II}
With Proposition~\ref{prop:compare}, we can deduce a formula for the archimedean zeta $I_\infty(k,\cD_{l_1,l_2},\cD_l,\Lambda_\infty)$ appearing in the interpolation formula in Theorem~\ref{thm:pFL1} and finish the proof of Theorem~\ref{thm:main}.

\begin{proof}[Proof of Theorem~\ref{thm:main}.]
We choose $K^p_{\GL(2)},K^p_{\GSp(4)}$ and primitive Hida families $\sB,\sB'$ of tame level $K^p_{\GL(2)}$ such that $\Spec(\bT_{\GSp(4),\ord})$ has an irreducible component $\theta(\sB,\sB')$. With such a choice, we can further choose $\bS$ and $\bLam$ such that $\cF_{\theta(\sB,\sB'),\bS,\bLam}\neq 0$. (By the interpolation property of $\cF_{\theta(\sB,\sB'),\bS,\bLam}$ in Proposition~\ref{prop:intBes}, to show the existence of such a $\bS$ and $\bLam$, it suffices to show that there exists $x$ satisfying the conditions there for which $B^\dagger_{\bS,\Lambda}(\varphi)\neq 0$ for some $\varphi\in \sS_{\GSp(4),x}$. Take an $x$ with corresponding weight $(l_1,l_2)$, $l_1\gg l_2\gg 0$ and $\varphi\in \sS_{\GSp(4),x}$. One can choose $\bS,\Lambda$ such that the usual Bessel period $B_{\bS,\Lambda}(\varphi)\neq 0$. Then by \cite[Proposition~2.7.1]{pFL}, we know that  $B^\dagger_{\bS,\Lambda}(\varphi)\neq 0$.) Then both sides of the identity in Proposition~\ref{prop:compare} are nonzero elements in $\cM_{\GSp(4),\ord}\otimes_{\wt{\Lambda}_{\GSp(4)}} F_{\theta(\sB,\sB')}$. (There are many interpolation points corresponding to $s$ belonging to the absolute convergence range, at which one can check that the evaluations are nonzero.)

Let $(\kappa,x)$ be a point of $\Hom_\cont\left(\bQ^\times\backslash\bA^\times_{\bQ,f}/U^p,\bar{\bQ}^\times_p\right)\times\sC_1(\bar{\bQ}_p)\times \theta(\sB,\sB')(\bar{\bQ}_p)$ as in Theorem~\ref{thm:pFL1} with $\sC_2=\theta(\sB,\sB')$. Then
\begin{align*}
\text{RHS}(\kappa,x)
   =&\,\bbc\sqrt{\det\bS}\,2^{-3}i^{-1} f_{x,\bbc}f_x\cdot 2^{-l-l_1-l_2-1}i^{l+1}
    \sum_{\varphi\in \sS_{\GSp(4),x}}\frac{\bes^\dagger_{\bS,\Lambda}(\varphi)\varphi}{\bfP(\varphi,\varphi)} \\
   &\times\frac{D^S\left(k+\frac{l+l_2+l_2}{2},\sigma_x\times\pi_x\times\chi\right) D^S\Big(k+\frac{l+l_1+l_2}{2},\pi_x\times\sigma'_x\times\omega^{\frac{1}{2}}_{\sigma_x}\omega^{-\frac{1}{2}}_{\sigma'_x}\times\chi\Big)}{\bfP(f_x,f_x)} 
\end{align*}
by the interpolation properties of $\cL^{S,\ast}_{\sB,\sC_1},\cL^{S,\ast}_{\sC_1,\sB'}, \cF_{\theta(\sB,\sB'),\bS,\bLam}$, and
\begin{align*}
   {\rm LHS}(\kappa,x)
   =&\, i^{r_{\sLambda,1}-r_{\sLambda,2}} f_{x,\bbc}f_x \cdot\frac{I_\infty(k,\cD_{l_1,l_2},\cD_l,\Lambda_\infty)}{E_\infty\Big(k+\frac{l+l_1+l_2}{2},\cD_{l_1,l_2}\times\cD_l\times\chi\Big)}\sum_{\varphi\in \sS_{\GSp(4),x}}\frac{\bes^\dagger_{\bS,\Lambda}(\varphi)\varphi}{\bfP(\varphi,\varphi)} \\
   &\times\frac{D^S\left(k+\frac{l+l_2+l_2}{2},\sigma_x\times\pi_x\times\chi\right) D^S\Big(k+\frac{l+l_1+l_2}{2},\pi_x\times\sigma'_x\times\omega^{\frac{1}{2}}_{\sigma_x}\omega^{-\frac{1}{2}}_{\sigma'_x}\times\chi\Big)}{\bfP(f_x,f_x)} 
\end{align*}
by Theorem~\ref{thm:pFL1}. It follows that
\begin{equation}\label{eq:arch-zeta}
   I_\infty(k,\cD_{l_1,l_2},\cD_l,\Lambda_\infty)= \bbc\sqrt{\det\bS}  \,2^{-l-l_1-l_2-4}i^{l-r_{\sLambda,1}+r_{\sLambda,2}} \cdot E_\infty\Big(k+\frac{l+l_1+l_2}{2},\cD_{l_1,l_2}\times\cD_l\times\chi\Big),
\end{equation}
for all $(k,l,l_1,l_2)$ which equals the algebraic part of the projection of an arithmetic $(\kappa,x)$ to the weight space such that $x$ is not a pole of either side of\eqref{eq:compare} and ${\rm RHS}(\kappa,x)$ and ${\rm LHS}(\kappa,x)$ are nonzero.

Since we have made the choices such that both sides of \eqref{eq:compare} are nonzero, the points for which the weight projection is not \'{e}tale at $x$ or ${\rm LHS}(\kappa,x)={\rm LHS}(\kappa,x)=0$ are not Zariski dense. For any $(k,l,l_1,l_2)$ satisfying \eqref{eq:kl}, the classical points $(\kappa,x)$ whose projections to the weight space has algebraic part equal to $(k,l,l_1,l_2)$ are Zariski dense, so there exist $(\kappa,x)$ which satisfies the conditions for the above comparison to deduce \eqref{eq:arch-zeta} for the given $(k,l,l_1,l_2)$. Thus, \eqref{eq:arch-zeta} is true for all $(k,l,l_1,l_2)$ satisfying \eqref{eq:kl}. 

Plugging it into the interpolation formula in Theorem~\ref{thm:pFL1} shows that 
\[
   \cL^S_{\sC_1,\sC_2,\beta_1,\beta_2}=(\bbc\sqrt{\det\bS})^{-1}2^4\cdot \varepsilon_{\qexp,\beta_1,\beta_2}\left(\mu^S_{\sC_1,\sC_2}\right)
\]
is the desired $p$-adic $L$-function.
\end{proof}

\begin{rmk}
With \eqref{eq:kl}, the interpolation formula for the one-variable cyclotomic $p$-adic $L$-function $\cL^S_{\Pi,\pi}$ in \cite[Theorem~1.0.1]{pFL} becomes
\begin{align*}
   &\cL^S_{\Pi,\pi}\left((\chi|\bdot|^k)_{p_\adic}\right)
   = \,\bbc\sqrt{\det\bS}\, 2^{-l-l_1-l_2-4}i^{l-r_{\sLambda,1}+r_{\sLambda,2}}\cdot \frac{\bes^\dagger_{\bS,\Lambda}\left(\varphi_{\ord}\right)
   \,\whi_\bbc(f_\ord)}{ \bfP(\varphi_\ord,\varphi_\ord)\,\bfP(f_\ord,f_\ord)} \\[1ex]
   &\times 
   \left\{\begin{array}{ll}
    E_\infty\left(k,\tilde{\Pi}\times\tilde{\pi}\times\chi\right)
   E_p\left(k,\tilde{\Pi}\times\tilde{\pi}\times\chi\right)
   \cdot  L^S\left(k,\tilde{\Pi}\times\tilde{\pi}\times\chi\right), &\text{$l_1+l_2+l$ even},\\[2ex] 
    E_\infty\left(k+\frac{1}{2},\tilde{\Pi}\times\tilde{\pi}\times\chi\right)
   E_p\left(k+\frac{1}{2},\tilde{\Pi}\times\tilde{\pi}\times\chi\right)
   \cdot  L^S\left(k+\frac{1}{2},\tilde{\Pi}\times\tilde{\pi}\times\chi\right), &\text{$l_1+l_2+l$ odd},
   \end{array}\right.
\end{align*}
with $(k,\chi)$ satisfying the conditions in {\it loc.cit}.

\end{rmk}

\bibliographystyle{halpha}
\bibliography{BiStd}
\end{document}